\documentclass[12pt,a4paper]{article}

\usepackage[leqno]{amsmath}
\usepackage{amssymb,amsthm,upref,amscd}
\usepackage[T1]{fontenc}
\usepackage{times}
\usepackage{cite}
\usepackage{amsfonts}
\usepackage[colorinlistoftodos]{todonotes}
\usepackage{setspace}

\usepackage{color}
\usepackage{hyperref}
\usepackage{graphicx}
\usepackage{wrapfig}
\usepackage{setspace}

\usepackage[usenames,dvipsnames]{pstricks}
\usepackage{epsfig}

\newtheorem{thm}{Theorem}[section]
\newtheorem{lem}[thm]{Lemma}

\newtheorem*{ack}{Acknowledgments}

\theoremstyle{definition}

\theoremstyle{remark}
\newtheorem*{notation}{Notation}
\theoremstyle{remark}
\newtheorem{remark}[thm]{Remark}
\numberwithin{equation}{section}

\newcommand{\C}{{\mathbb C}}

\newcommand{\R}{{\mathbb R}}

\definecolor{blue}{rgb}{0,0,1}

\newcommand{\al}{\alpha}
\newcommand{\be}{\beta}

\newcommand{\de}{\delta}
\newcommand{\la}{\lambda}

\newcommand{\De}{\Delta}

\newcommand{\vphi}{\varphi}
\newcommand{\eps}{\varepsilon}
\newcommand{\pa}{\partial}

\newcommand{\beq[1]}{\begin{equation}\label{eq:#1}}
\newcommand{\eeq}{\end{equation}}

\begin{document}

\title{Multiple positive normalized solutions for nonlinear Schr\"odinger systems}


\author{\sc{Tianxiang Gou}
\and
\sc{Louis Jeanjean}}

\date{}
\maketitle

\begin{abstract}
We consider the existence of multiple positive solutions to the nonlinear Schr\"odinger systems set on $H^1(\R^N) \times H^1(\R^N)$,
\[
\left\{
\begin{aligned}
-\De u_1 &= \la_1 u_1 + \mu_1 |u_1|^{p_1 -2}u_1
            + \be r_1 |u_1|^{r_1-2}u_1|u_2|^{r_2}, \\
-\De u_2 &= \la_2 u_2 + \mu_2 |u_2|^{p_2 -2}u_2
            +  \be r_2 |u_1|^{r_1}|u_2|^{r_2 -2}u_2,
\end{aligned}
\right.
\]
under the constraint
\[
\int_{\R^N}|u_1|^2 \, dx = a_1,\quad \int_{\R^N}|u_2|^2 \, dx = a_2.
\]
Here $a_1, a_2 >0$ are prescribed, $\mu_1, \mu_2, \beta>0$, and the frequencies $\lambda_1, \lambda_2$ are unknown and will appear as Lagrange multipliers. Two cases are studied, the first when $N \geq 1, 2 < p_1, p_2 < 2 + \frac 4N, r_1, r_2 > 1, 2 + \frac 4N < r_1 + r_2 < 2^*$, the second when $ N \geq 1, 2+ \frac 4N < p_1, p_2 < 2^*,  r_1, r_2 > 1, r_1 + r_2 < 2 + \frac 4N.$ In both cases, assuming that $\beta >0$ is sufficiently small, we prove the existence of two positive solutions. The first one is a local minimizer for which we establish the compactness of the minimizing sequences and also discuss the orbital stability of the associated standing waves. The second solution is obtained through a constrained mountain pass and a constrained linking respectively.
\end{abstract}


\section{Introduction}\label{intro}

In this paper, we are concerned with the existence of normalized solutions to some nonlinear Schr\"odinger systems. More precisely for $a_1>0, a_2 >0$ given, we look for the existence of $(\lambda_1, \lambda_2, u_1, u_2) \in \R^2 \times H^1(\R^N) \times H^1(\R^N)$ satisfying
\begin{equation}\label{sys1}
\left\{
\begin{aligned}
-\De u_1 &= \la_1 u_1 + \mu_1 |u_1|^{p_1 -2}u_1
            + \be r_1 |u_1|^{r_1-2}u_1|u_2|^{r_2}, \\
-\De u_2 &= \la_2 u_2 + \mu_2 |u_2|^{p_2 -2}u_2
            + \be  r_2 |u_1|^{r_1}|u_2|^{r_2 -2}u_2,
\end{aligned}
\right.
\end{equation}
and
\begin{equation}\label{sys2}
\int_{\R^N}|u_1|^2 \, dx = a_1,\quad \int_{\R^N}|u_2|^2 \, dx = a_2.
\end{equation}
Throughout the paper we asume that $\beta >0$ and $\mu_i >0$ for $i=1,2$. Various assumptions on $N$ and $p_i,r_i$ for $i = 1,2$ will be introduced but we shall always require $r_i >1, 2 < p_i <2^*$ for $i=1,2$ and $r_1 + r_2 < 2^*$.

The study of system \eqref{sys1} where $(\lambda_1, \lambda_2)$ are given has been, starting from the { pioneering} papers \cite{ACo, BuSi, LWe, MMP, Si},  the subject of a huge {literature} in the recent years. In our situation note that since \eqref{sys2} is added to \eqref{sys1}, namely the masses $a_1,a_2$ are prescribed, the frequencies $\lambda_1, \lambda_2$ are necessarily unknown of the problem. They will appear as Lagrange multipliers.  Actually the approches to solve \eqref{sys1} with $(\lambda_1, \lambda_2)$ given or to solve \eqref{sys1}-\eqref{sys2} turn out to be quite distinct.

The problem under consideration arises from the search of normalized standing waves to the following nonlinear Schr\"odinger systems
\begin{equation}\label{sys}
\begin{cases}
- i \pa_t \Psi_1 = \De \Psi_1 + \mu_1 |\Psi_1|^{p_1-2}\Psi_1
                   +\be r_1 |\Psi_1|^{r_1-2}\Psi_1|\Psi_2|^{r_2},\\
- i \pa_t \Psi_2 = \De \Psi_2 + \mu_2 |\Psi_2|^{p_2-2}\Psi_2
                   +\be r_2|\Psi_1|^{r_1}|\Psi_2|^{r_2-2}\Psi_2,
\end{cases} \text{in $\R \times \R^N$}
\end{equation}
modeling the Bose-Einstein condensates with multiple states, or the propagation of mutually incoherent waves packets in nonlinear optics, see \cite{AA, EGBB, F, T}. By standing waves we intend solutions to \eqref{sys} of the form
\begin{align*}
\Psi_1(t,x) = e^{-i\la_1 t} u_1(x), \quad \Psi_2(t,x) = e^{-i\la_2 t} u_2(x),
\end{align*}
for $(\lambda_1, \lambda_2) \in \R^2$, and $(u_1, u_2) \in  H^1(\R^N) \times H^1(\R^N)$. Clearly $(\Psi_1, \Psi_2)$ is a standing wave of \eqref{sys} if and only if $(\lambda_1, \lambda_2, u_1, u_2)$ is a solution \eqref{sys1}. Since the mass
is conserved along the trajectories of \eqref{sys}, i.e.
\[
\int_{\R^N} |\Psi_i(t,x)|^2 \, dx = \int_{\R^N} |u_i(x)|^2 \, dx \quad  \mbox{for all } t>0,
\]
for $i =1, 2,$ the study of the existence of normalized solutions is particularly relevant from a physical point of view.

Solutions to \eqref{sys1}-\eqref{sys2} will be obtained as critical points of the energy functional $J: H^1(\R^N) \times H^1(\R^N) \mapsto \R$ defined by
$$
J(u_1,u_2)
 := \frac12\int_{\R^N} |\nabla u_1|^2 + |\nabla u_2|^2 \, dx
     - \sum_{i = 1}^2 \frac{\mu_i}{p_i}\int_{\R^N} |u_i|^{p_i} \, dx
    -\beta \int_{\R^N} |u_1|^{r_1}|u_2|^{r_2} \, dx
$$
on the constraint $S(a_1,a_2):=S(a_1) \times S(a_2)$, where for any $a>0,$
$$
S(a):= \{u \in H^1(\R^N): \int_{\R^N}|u|^2 \, dx= a\}.
$$
Under the assumption $ N \geq 1, r_i >1,  2 <p_i < 2 + \frac{4}{N}$ for $i=1,2$ and $ r_1 + r_2 < 2 + \frac{4}{N}$
it is standard to show that $J$ is bounded from below and coercive on $S(a_1,a_2)$. Then one may search for a critical point of $J$ as a global minimizer of $J$ on $S(a_1,a_2)$ and more generally study the compactness of the associated minimizing sequences. It is also well known, under the assumption that the Cauchy problem is locally well posed, that if such sequences are, up to translation, compact then the set of global minimizers is orbitally stable. In that direction there had been a good amount of works, directly on \eqref{sys1}-\eqref{sys2} or on related problems, see in particular \cite{Bh,CCWei,NWa,Oh}, and the more complete result was recently obtained in \cite{GJ}.

On the contrary, when either $p_i > 2 + \frac{4}{N}$ for some $i=1,2$ or $r_1+r_2 > 2 + \frac{4}{N}$, the functional $J$ becomes unbounded from below on $S(a_1,a_2)$. To see this let us introduce for $t>0$ and $u \in H^1(\R^N)$ the scaling $u^t(x):=t^{\frac N2}u(tx)$. A direct calculation shows that if $u \in S(a)$, then $u^t \in S(a)$. Now the claim follows, see also Lemma \ref{negative}, by considering for an arbitrary $(u_1,u_2) \in S(a_1,a_2)$ the map (i.e. the dilation)
\begin{align}\label{underscale}
\begin{split}
t \mapsto J(u_1^t, u_2^t) &=  \frac{t^2}{2} \int_{\R^N} |\nabla u_1|^2 + |\nabla u_2|^2 \, dx
            -\sum_{i =1}^2  t^{(\frac{p_i}{2}-1)N}\frac{\mu_i}{p_i}  \int_{\R^N}|u_i|^{p_i} \, dx\\
            &-\beta  t^{(\frac{r_1+r_2}{2}-1)N}\int_{\R^N} |u_1|^{r_1}|u_2|^{r_2} \, dx.
\end{split}
\end{align}
When a global minimizer { does not exist}, finding a critical point for $J$ on $S(a_1,a_2)$ is more involved. In \cite{BJ} it is assumed that $N \leq 4$ with $2 <p_1 < 2 + \frac{4}{N} < p_2 < 2^*$, $2 + \frac{4}{N}< r_1 + r_2 < 2^*$, $r_2 >2$ and, under some restrictions on $(a_1,a_2)$, a solution $(\lambda_1, \lambda_2, u_1, u_2)$ of \eqref{sys1}-\eqref{sys2} with $\lambda_i <0$ and $u_i >0$ for $i=1,2$ is found.  In \cite{BJS} still for $N \leq 4$ assuming that $2 + \frac{4}{N} <p_i < 2^*$ and $2 + \frac{4}{N}< r_1 =r_2 < 2^*$ a solution is obtained either when $\beta >0$ is sufficiently small or sufficiently large. Let us also mention the recent work \cite{BS} where, under the condition that $\beta <0$, a positive solution is obtained under quite general assumptions, see also \cite{BS2} for a multiplicity result.

Actually the search of normalized solutions for functionals which are unbounded from below on the constraint and presents a lack of compactness (typically associated with the fact that the underlying equation or system is set on all $\R^N$) is still a widely unexplored field. In the case where the associated equation is autonomous and set on $\R^N$ we refer to \cite{BJ,BJS,BS,Jeanjean,BJL,BJ2,Bv}. Let us also mention the recent papers \cite{NTV,NoTaVe,PV} dealing with situations where the equation is set on a bounded domain or when a trapping potential is acting. The difficulties and techniques introduced in these last works  differ however significantly. On one hand more compactness is available, on the other hand it is not possible to use the dilations which play an essential role in \cite{BJ,BJS,BS,Jeanjean,BJL,BJ2,Bv}.

In this paper we extend the results of \cite{BJ,BJS,GJ,BS} as to cover two new ranges of parameters. Our assumptions are
\begin{itemize}
\item[$(H_0)$] $N \geq 1,$ $2 < p_1, p_2 < 2 + \frac{4}{N},$ $r_1, r_2 >1,$ $ 2 + \frac{4}{N}  < r_1 + r_2 < 2^*$ .
\end{itemize}
\begin{itemize}
\item[$(H_1)$] $ N \geq 1,$ $2 + \frac 4N < p_1, p_2 < 2^*,$ $ r_1, r_2 >1,$ $ r_1 + r_2 < 2 + \frac 4N.$
\end{itemize}
Our aim is to prove, assuming $\beta >0$ sufficiently small, that there exists under $(H_0)$ or $(H_1)$ two distinct positive solutions to \eqref{sys1}-\eqref{sys2}, namely solutions  $(\lambda_1, \lambda_2, u_1,u_2)$ where $u_1>0$ and $u_2>0$. Up to our knowledge it is the first time a multiplicity result for positive solutions is obtained for \eqref{sys1}-\eqref{sys2} when $\beta >0$. Note however that, when $\beta <0$ the existence of infinitely many positive solutions was established in \cite{BS2}.

In order to state our results let us introduce, for $\rho >0,$
$$
\mathcal{B}(\rho):=\{(u_1, u_2) \in H^1(\R^N)\times H^1(\R^N): \int_{\R^N}|\nabla u_1|^2 + |\nabla u_2|^2 \, dx < \rho\}.
$$
We shall prove that under either $(H_0)$ or $(H_1)$ for any $\rho >0$
\begin{equation}\label{i1}
\inf J(u_1,u_2) <0 \quad \mbox{for} \quad (u_1,u_2) \in S(a_1,a_2) \cap \mathcal{B}(\rho)
\end{equation}
and that there exists a $\beta_0 = \beta_0(a_1,a_2) >0$ and $\rho_0 = \rho_0(a_1,a_2) >0$ such that
\begin{equation}\label{i2}
\inf J(u_1,u_2) >0 \quad \mbox{for} \quad (u_1,u_2) \in S(a_1,a_2) \cap \partial \mathcal{B}(\rho_0)
\end{equation}
for any $0 < \beta \leq \beta_0$.

{Together \eqref{i1} and \eqref{i2}, suggest that $J$ may have, on $S(a_1,a_2)$, a  local minimum} and thus it is natural, for $0 < \beta \leq \beta_0$, to introduce the minimization problem
\begin{equation}\label{i3}
m(a_1,a_2) := \inf J(u_1, u_2) <0  \quad \mbox{for} \quad (u_1, u_2) \in S(a_1, a_2) \cap \mathcal{B}(\rho_0).
\end{equation}
We shall prove that any minimizing sequence for \eqref{i3} is, up to translation, compact and in particular this will imply the existence of a positive solution of \eqref{sys1}-\eqref{sys2} at this energy level. To prove the compactness of the minimizing sequences, instead of trying directly to check the strict inequalities proposed by P. L. Lions \cite{Li1,Li2} we make use of the rearrangement introduced by M. Shibata \cite{Sh2} as presented in \cite[Lemma A.1]{Ik}. This was already the approach in  \cite{GJ} but here we need to adapt it to the case where the {global minimum} is replaced by a local one.  A new difficulty arises from the fact that in general the sum of two elements in $\mathcal{B}(\rho_0)$ does not belong to $\mathcal{B}(\rho_0)$ and this makes things more technical when discussing a possible dichotomy. Note that a similar difficulty was recently encountered in \cite{BBJV} but we propose here an alternative approach.

As already observed $(H_0)$ or $(H_1)$ imply that for any $(u_1,u_2) \in S(a_1,a_2)$, $J(u_1^t,u_2^t) \to - \infty$ as $t \to \infty$ and $(u_1^t,u_2^t) \notin \mathcal{B}(\rho_0)$ for $t>0$ large enough. This property along with \eqref{i1}-\eqref{i2} suggests that there { may exist} other critical points. Actually under $(H_0)$ a second positive solution will be obtained by a mountain-pass argument while under $(H_1)$ a linking type procedure, inspired by \cite{BJS}, will be used. Let us now state our main results.
\begin{thm}\label{thm1}
Let $a_1, a_2 >0$ be given and assume that $(H_0)$ holds. Then there exist $\beta_0=\beta_0(a_1, a_2) >0$ and $\rho_0 = \rho_0(a_1, a_2)>0$ such that, for any $0 < \beta \leq \beta_0$,
\begin{itemize}
\item[(i)] every minimizing sequence of \eqref{i3} is compact, up to translation, in $H^1(\R^N) \times H^1(\R^N)$. In particular,
{there exists a solution $(\lambda_1^v, \lambda_2^v,v_1, v_2)$ to \eqref{sys1}-\eqref{sys2} with $(v_1, v_2) \in \mathcal{B}(\rho_0)$ having both components positive and such that $J(v_1, v_2) <0.$}
\item[(ii)]  If $2 \leq N \leq 4$ or  $N \geq 5$ with either
         $p_i \leq r_1 + r_2 - \frac{2}{N}$, $i=1,2$ or $|p_1-p_2| \leq \frac{2}{N}$, {there exists a second solution $(\lambda_1^u, \lambda_2^u, u_1, u_2)$
to \eqref{sys1}-\eqref{sys2} with  $(u_1, u_2)$ having both components positive and such that  $J(u_1, u_2)>0$.}
\end{itemize}
{Moreover, both $\lambda_1^v, \lambda_2^v <0$ and  $\lambda_1^u, \lambda_2^u <0$  if $N \leq 4$}.
\end{thm}

\begin{thm}\label{thm2}
Let $a_1, a_2 >0$ be given and assume that $(H_1)$ holds. Then there exist $\beta_0=\beta_0(a_1, a_2) >0$ and $\rho_0 = \rho_0(a_1, a_2)>0$ such that, for any $0 < \beta \leq \beta_0$,
\begin{itemize}
\item[(i)] {If $1 \leq N \leq 4$ or $N \geq 5$ with ${r_i \geq \big(\frac{r_1 +r_2}{2}-1\big) N}$ for $i=1,2$,} every minimizing sequence of \eqref{i3} is compact, up to translation, in $H^1(\R^N) \times H^1(\R^N)$. In particular, {there exists a solution $(\lambda_1^v, \lambda_2^v,v_1, v_2)$ to \eqref{sys1}-\eqref{sys2} with $(v_1, v_2) \in \mathcal{B}(\rho_0)$ having both components positive and such that $J(v_1, v_2) <0.$}
\item[(ii)] If $2 \leq N \leq 4$, there exists a second solution {$(\lambda_1^u, \lambda_2^u, u_1, u_2)$ to \eqref{sys1}-\eqref{sys2} with  $(u_1, u_2)$ having both components positive and such that  $J(u_1, u_2)>0$.}
\end{itemize}
{ Moreover, both  $\lambda_1^v, \lambda_2^v <0$ and  $\lambda_1^u, \lambda_2^u <0$  if $N \leq 4$}.
\end{thm}

\begin{remark} $ $
i) The value of $\beta_0 >0$ in Theorems \ref{thm1} and \ref{thm2} can be explicitly computed in terms of  $N,p_i,a_i,r_i$. Our results are not perturbative. In addition for any $\beta >0$ we can assume that $\beta \leq \beta_0$ at the expense of taking $a_1 >0$ and $a_2 >0$ sufficiently small because $\beta_0(a_1, a_2) \to \infty$ as $a_1, a_2 \to 0$, for more details, see Lemma \ref{mpgeo1}.

ii) The existence of the solution in Theorem \ref{thm1}-\ref{thm2} $(ii)$ is under the condition $N \geq 2$.  This is because we search it in the {radially symmetric} space $H_{rad}^1(\R^N) \times H^1_{rad}(\R^N)$ where the compact embedding $H^1_{rad}(\R^N)\hookrightarrow L^p(\R^N)$ for $2 <p < 2^*$ holds only when $N \geq 2$.

iii)  We conjecture that Theorem \ref{thm1} (ii) is true assuming only $(H_0)$ and we refer to Remark \ref{extension} for a discussion in that direction.
\end{remark}

\begin{remark}\label{symetry}
{When $N \geq 2,$ by  Schwarz-symmetrization arguments it is possible to obtain the additional property that solutions obtained in Theorems \ref{thm1} and \ref{thm2} are radially symmetric and decreasing with respect to a common point. More generally it could be of interest to know if any positive solution to \eqref{sys1}-\eqref{sys2} has this property.  Assuming that $\lambda_1$ and $\lambda_2$ are strictly negative and that  $r_1, r_2 \geq 2$  this likely follows from \cite[Theorem 1]{BuSi}. However the general case seems to be open.}
\end{remark}

\begin{remark}\label{perturbative}
{ By Theorem \ref{thm1}, for any $0 < \beta < \beta_0$, there exists a solution $(\lambda_1^{\beta}, \lambda_2^{\beta}, v_1^{\beta},v_2^{\beta})$  to \eqref{sys1}-\eqref{sys2} where $(v_1^{\beta},v_2^{\beta})$ is a local minimizer to \eqref{i3}. Adapting some arguments of the proof of Theorem \ref{thm1} it can be shown that $(v_1^{\beta},v_2^{\beta})$ converges towards
$(v_1^0, v_2^0)$ in $H^1(\R^N) \times H^1(\R^N)$ as $\beta \to 0^+$. Furthermore, for $i=1,2$, $v_i^0$ satisfies  the equation
    \begin{align} \label{equ1}
    -\Delta v_i^0 =\lambda_i^0 v_i^0 + |v_i^0|^{p_i-2} v_i^0.
    \end{align}
for some $\lambda_i^0 <0$. Thus the existence result of Theorem \ref{thm1} could likely be obtained as a consequence of the Implicit Function Theorem. However, as already noted, Theorem \ref{thm1} is not perturbative and also our second solution cannot, of course, be obtained in this way. Indeed, as $\beta \to 0^+$ it is readily checked that the norm to $(u_1^{\beta},u_2^{\beta})$ goes to infinity. The same comments  hold concerning Theorem \ref{thm2} but this time reversing the  role of the solutions $(\lambda_1^{v}, \lambda_2^{v}, v_1,v_2)$ and $(\lambda_1^{u}, \lambda_2^{u}, u_1,u_2).$}
\end{remark}

The proofs of Theorems \ref{thm1}(ii) and \ref{thm2}(ii) follows the general strategy developed in the papers \cite{BJ,BJS,BS} which all deal with the search of constrained critical points which are not minimizers. First one needs to identify a possible critical level. This is done by introducing a minimax structure of mountain pass type when $(H_0)$ and of linking type under $(H_1)$. Secondly one has to show that there exists a bounded Palais-Smale sequence, say $\{(u_1^n,u_2^n)\} \subset S(a_1,a_2)$ at this energy level. This step relies, as in \cite{BJ,BJS,BS}, on the presence of a natural constraint of Pohozaev type on which the functional is coercive. To take advantage of this constraint we rely here on the approach, first introduced in \cite{Jeanjean}, which consists in the addition of an artificial variable directly within the variational procedure. It is now a standard technique in problems where dilatations are allowed. At this point one can assume that $(u_1^n,u_2^n) \rightharpoonup (u_1,u_2)$ weakly in $H^1(\R^N) \times H^1(\R^N)$. The last step consists in showing the strong convergence. The key point being the convergence in $L^2(\R^N) \times L^2(\R^N) $ since we need ($u_1,u_2) \in S(a_1,a_2).$ It is this step which induces the main limitations in \cite{BJ,BJS,BS}. To insure the strong convergence one relies on the use of { a Liouville's type result}, see Lemma \ref{liouville}, which only applies when $N \leq 4$. Also the proofs in \cite{BJ,BJS,BS} use the property that the scalar problem
\begin{align}\label{eqw0}
-\Delta w - \lambda w = \mu |w|^{p-2}w,
\end{align}
has a unique positive solution $u \in S(a)$ for $2 <p<2^*, \mu>0$.

In this paper we start to relax  these two restrictions. Theorem \ref{thm1} allows to consider the case $N \geq 5$ and no direct uniqueness assumption is imposed on \eqref{eqw0}. Our second solution is found through a mountain-pass argument, more precisely we first prove that, for any $0 < \beta < \beta_0$, there exists a $ 0 < \bar{\rho} = \bar{\rho}(a_1,a_2) < \rho_0$ such that
\begin{align*}
\gamma(a_1, a_2): = \inf_{g \in \Gamma} \max_{t \in [0, 1]} J(g(t))> \max \{J(g(0)), J(g(1))\},
\end{align*}
where
\begin{align*}
\Gamma := \{g \in C([0, 1], S(a_1,a_2)): g(0) \in \mathcal{B}(\bar{\rho}) , g(1) \notin \overline{{\mathcal{B}(\rho_0)}} \mbox{ with } J(g(1)) < 0\}.
\end{align*}
Having obtained a bounded Palais-Smale sequence at the level $\gamma(a_1,a_2)$ we denote by $(u_1,u_2)$ its weak limit. An appropriate choice of the Palais-Smale sequence insures that
\begin{align}\label{iden}
J(u_1, u_2) \leq \gamma(a_1,a_2).
\end{align}
When $N \leq 4$ the fact that $(u_1, u_2) \in S(a_1,a_2)$ is directly obtained by the Liouville's argument. When $N \geq 5$  we argue by contradiction. If
either $\bar{a_1}:= ||u_1||_2^2 < a_1$ or $\bar{a_2}:= ||u_2||_2^2 < a_2$ we manage to construct a path $g \in \Gamma$ on which the maximum of $J$ is strictly below $J(u_1,u_2)$. By the characterization of $\gamma(a_1,a_2)$ we thus get
$$
\gamma(a_1, a_2) \leq \max_{0 \leq t \leq 1}J(g(t)) < J(u_1, u_2),
$$
in contradiction with \eqref{iden}. The construction of this path $g \in \Gamma$ relies on the property that for
$2 < p < 2 + \frac 4N$
\begin{equation}\label{i4}
 \inf_{u \in S(a)}I(u) <0, \quad \mbox{where} \quad I(u):= \frac 12 \int_{\R^N}|\nabla u|^2 \,dx - \frac{\mu}{p} \int_{\R^N}|u|^p \, dx.
\end{equation}
Starting from \eqref{iden} and under the assumptions of Theorem \ref{thm1}(ii), we first construct a path $\overline{g}$ lying in $S(\bar{a_1},\bar{a_2})$ on which the maximum of $J$  is below $\gamma(a_1,a_2)$. Then thanks to \eqref{i4} we transform $\overline{g}$ into a path $g \in \Gamma$ which satisfies
$$\max_{t \in [0,1]}J(g(t)) < \max_{t \in [0,1]}J(\bar{g}(t)) \leq J(u_1,u_2)$$
by ``adding some masses" somehow in the spirit of \cite{JS} but using here again the rearrangement results of \cite{Sh2}.

In Theorem \ref{thm2} (ii), to look for a second solution, we establish a linking structure for $J$ restricted to the constraint.  Since $p_i > 2 + \frac{4}{N}$ for $i=1,2$, \eqref{i4} does not hold and our proof relies as in \cite{BJ, BJS, BS} on the Liouville argument inducing the restriction $N \leq 4$. However in Theorem \ref{thm2} (i) we manage to consider situations where $N \geq 5$ at the expense of a restriction on the range of $r_1$ and $r_2$.

We now set
$$
G(a_1,a_2) := \{(u_1, u_2) \in S(a_1, a_2) \cap \mathcal{B}(\rho_0): J(u_1, u_2) = m(a_1, a_2)\}.
$$
Note that under the assumptions $(H_0)$ or $(H_1)$  it is unknown if \eqref{sys} is locally well posed. The point being that when $1 <r_i < 2$ for  $i=1,2$ the interaction part is not { Lipschitz}
continuous and in particular the uniqueness may fail. For a general discussion in that direction we refer to \cite{NgTiDeSh}. As a consequence our last result which states the orbital stability  of the set of standing waves associated to $G(a_1,a_2)$ is only valid {under the condition that the local existence of the Cauchy problem to \eqref{sys} holds}. Since its proof, having at hand the compactness of the minimizing sequences, relies on the classical arguments of \cite{CL}, we do not provide it.

\begin{thm}\label{th:sta}
{Assume that the local existence of the Cauchy problem in \eqref{sys} holds. Then under the assumptions of Theorems \ref{thm1}(i) and \ref{thm2}(i)} the set $G(a_1, a_2)$ is orbitally stable, i.e. for any $\eps>0$, there exists $\de>0$
so that if the initial condition $(\psi_1(0), \psi_2(0))$ in system \eqref{sys} satisfies
\begin{equation*}
\inf_{(u_1, u_2)\in G(a_1, a_2)}
\|(\psi_1(0), \psi_2(0)) - (u_1,u_2)\|\leq \de,
\end{equation*}
then
\begin{equation*}
\sup_{t \geq 0} \inf_{(u_1, u_2) \in G(a_1, a_2)}
\|(\psi_1(t), \psi_2(t)) - (u_1, u_2)\| \leq \eps,
\end{equation*}
where $(\psi_1(t), \psi_2(t))$ is the solution of system \eqref{sys}
corresponding to the initial condition $(\psi_1(0), \psi_2(0))$ and $\|\cdot\|$ denotes the norm in Sobolev space $H^1(\R^N) \times H^1(\R^N)$.
\end{thm}

Let us {point out} that we do know situations where the local existence holds. For example when $N=1$ and $r_1=r_2 :=r$ with $2 \leq r < 3$, see \cite{NgTiDeSh}.

Our Theorem \ref{th:sta} is a contribution to the recently studied question of proving that a local minimizer which is not global minimizer and for  a nonlinearity which is $L^2$-supercritical may lead to an orbitally stable standing wave. In that direction we are just aware of \cite{BBJV,BJ2}.

The paper is organized as follows. In Section \ref{pre}, we establish some preliminaries. Section \ref{pro1} is devoted to the proof of Theorems \ref{thm1}(i) and \ref{thm2}(i). In Section \ref{pro2} we give the proofs of Theorems \ref{thm1}(ii) and \ref{thm2}(ii). Finally in an appendix we establish a key technical result, Lemma \ref{lem124}.

\begin{ack}
{\rm This work has been carried out in the framework of the Project NONLOCAL (ANR-14-CE25-0013), funded by the French National Research Agency (ANR), T. Gou is supported by the China Scholarship Council.}
\end{ack}

\begin{notation}
In this paper we denote for any $1\leq p<\infty,$ by $L^p(\R^N)$  the usual Lebesgue space with norm
$\|u\|_p^p := \int_{\R^N}|u|^p\,dx,$
and by $H^1(\R^N)$ the usual Sobolev space endowed with the norm
$\|u\|^2 := \int_{\R^N}|\nabla u|^2+|u|^2 \, dx.$
We denote by $'\rightarrow'$ and $'\rightharpoonup'$ strong convergence and weak convergence, respectively.
\end{notation}


\section{Preliminary results}\label{pre}
First of all, observe that the energy functional $J$ is well-defined in $H^1(\R^N) \times H^1(\R^N)$ thanks to the H\"{o}lder inequality,
$$
\int_{\R^N} |u_1|^{r_1}|u_2|^{r_2} \, dx
   \le \|u_1\|_{r_1q}^{r_1}\|u_2\|_{r_2q'}^{r_2}
   < \infty,
$$
for some $1 < q < 2^*, q' =\frac{q}{q-1}$ with $2 \leq r_1 q, r_2 q' \leq 2^*$. Recalling the Gagliardo-Nirenberg inequality, for $u \in  H^1(\R^N)$, $2 \leq p \leq 2^* $,
\begin{align}\label{ga1}
\|u\|_p \le C(N,p)\|\nabla u\|_2^{\al(p)} \|u\|_2^{1-\al(p)}\quad
 \text{where } \al(p)=\frac{N(p-2)}{2p},
\end{align}
we get for $(u_1, u_2) \in S(a_1, a_2)$,
\begin{align} \label{ga2}
\begin{split}
\int_{\R^N} |u_1|^{r_1}|u_2|^{r_2} \, dx &
  \le \|u_1\|_{r_1q}^{r_1}\|u_2\|_{r_2q'}^{r_2} \\
 & \le C a_1^{\frac{(1- \al(r_1q))r_1}{2}}a_2^{\frac{(1- \al(r_2q))r_2}{2}}\|\nabla u_1\|_2^{\frac{N(r_1q-2)}{2q}}
  \|\nabla u_2\|_2^{\frac{N(r_2q'-2)}{2q'}}
\end{split}
\end{align}
with $C=C(N,r_1,r_2,q)$.

We now recall the rearrangement introduced by Shibata \cite{Sh2} as presented in \cite{Ik}. Let $u$ be a Borel measurable function on $\R^N$. It is said to vanish at infinity if the level set $|\{x \in \R^N: |u(x)|>t\}|< \infty$ for every $t>0$.
Here $|A|$ stands for the $N$-dimensional Lebesgue measure of a Lebesgue mesurable  set $A \subset \R^N$.
Considering two Borel mesurable functions $u,v$ which vanish at infinity in $\R^N$, we define for $t>0$,
$A^{\star}(u, v; t):= \{x \in \R^N: |x|<r\}$ where $r>0$ is chosen so that
$$
|B(0, r)| = |\{x \in \R^N: |u(x)|>t\}| + |\{x \in \R^N: |v(x)|>t\}|,
$$
and $\{u, v\}^{\star}$ by
\begin{equation}\label{rearr}
\{u, v\}^{\star}(x):= \int_{0}^{\infty} \chi_{A^{\star}(u, v; t)} (x) \, dt,
\end{equation}
where $\chi_A(x)$ is a characteristic function of the set $A \subset \R^N$.

\begin{lem} \cite [Lemma A.1]{Ik} \label{Ikoma}
\begin{itemize}
\item[(i)] The function $\{u, v\}^{\star}$ is radially symmetric, non-increasing and lower semi-continuous.
           Moreover, for each $t > 0$ there holds $\{ x \in \R^N : \{u, v\}^{\star} > t\} = A^{\star}(u, v; t)$.
\item[(ii)] Let $\Phi : [0, \infty) \rightarrow [0, \infty)$ be non-decreasing, lower semi-continuous, continuous at $0$
           and $\Phi(0) = 0$. Then $\{\Phi(u),\Phi(v)\}^{\star} = \Phi(\{u, v\}^{\star})$.
\item[(iii)] $\|\{u, v\}^{\star}\|_p^p = \|u\|_p^p + \|v\|_p^p$ \, for $ 1 \leq p < \infty$.
\item[(iv)] If $u, v \in H^1(\R^N)$, then $\{u, v\}^{\star} \in  H^1(\R^N)$
            and $\|\nabla \{u ,v\}^{\star}\|_2^2 \leq \|\nabla u\|_2^2 + \|\nabla v\|_2^2$.
            In addition, if $u, v \in (H^1(\R^N) \cap C^1(\R^N)) \setminus \{0\}$ are radially symmetric, positive and non-increasing, then
            $$
            \int_{\R^N} |\nabla\{u ,v\}^{\star}|^2 \, dx < \int_{\R^N} |\nabla u|^2 + \int_{\R^N} |\nabla v|^2\, dx.
            $$
\item[(v)] Let $u_1, u_2, v_1, v_2 \geq 0$ be Borel measurable functions which vanish at infinity, then
            $$
            \int_{\R^N} (u_1u_2 + v_1v_2) \,dx \leq \int_{\R^N} \{u_1, v_1\}^{\star}\{u_2, v_2\}^{\star} \, dx.
            $$
\end{itemize}
\end{lem}

As an application of Lemma \ref{Ikoma}, we obtain the following result.

\begin{lem} \label{Ikoma1}
 Assume $r_1, r_2>1,$ and $r_1+r_2 < 2^*$. Let $u_1, u_2, v_1, v_2$ be Borel measurable functions which vanish at infinity, then
$$
   \int_{\R^N}|u_1|^{r_1}|u_2|^{r_2} + |v_1|^{r_1}|v_2|^{r_2} \, dx
   \leq \int_{\R^N}\left(\{|u_1|, |v_1|\}^{\ast} \right)^{r_1} \left(\{|u_2|, |v_2|\}^{\ast}\right)^{r_2} \, dx.
$$
\end{lem}
\begin{proof}
In view of the property $(v)$ in Lemma \ref{Ikoma}, to prove the lemma it suffices to show that
$$
\left(\{|u|, |v|\}^{\ast} \right)^{r} (x) =\{|u|^r, |v|^r\}^{\ast} (x), \,\, \mbox{for any} \,\, x \in \R^N,
$$
where $r>1$, and $u, v$ are Borel measurable functions which vanish at infinity. For any fixed $x \in \R^N$, in view of the definition \eqref{rearr}, by making the change of variable and using the property $(i)$ in Lemma \ref{Ikoma}, we have that
\begin{align*}
\{|u|^r, |v|^r\}^{\ast}(x)&=\int_{0}^{\infty} \chi_{A^{\star}(|u|^r, |v|^r; t)} (x) \, dt
                          =\int_{0}^{\infty} \chi_{A^{\star}(|u|, |v|; t^{\frac{1}{r}})} (x) \, dt \\
                          &=r \int_{0}^{\infty} \chi_{A^{\star}(|u|, |v|; t)} (x) t^{r-1} \, dt \\
                          &= r \int_{0}^{\infty} \chi_{\{y \in \R^N: \{|u|, |v|\}^{\ast}(y) >t\}} (x)  t^{r-1} \, dt \\
                          &= r \int_{0}^{\{|u|, |v|\}^{\ast}(x)} t^{r-1} \, dt \\
                          &=\left(\{|u|, |v|\}^{\ast} \right)^{r}(x).
\end{align*}
This  ends the proof.
\end{proof}

\begin{lem}\cite[Lemma A.2]{Ik} \label{liouville}
Suppose $ p \in ]1,\frac{N}{N-2}]$ when $N \geq 3$ and $p \in]1,\infty[$ when $N=1,2$. Let $u\in L^p(\R^N)$ be a smooth nonnegative function satisfying $-\De u \ge 0$ in $\R^N$. Then $u \equiv 0$ .
\end{lem}

\begin{lem}\label{Lieb}
Assume $r_1, r_2 > 1, r_1 + r_2 \leq 2^* .$ If $(u_1^n, u_2^n)\rightharpoonup (u_1, u_2)$ in $H^1(\R^N)\times H^1(\R^N),$
then up to a subsequence,
$$
\int_{\R^N}|u_1^n|^{r_1}|u_2^n|^{r_2} - |u_1^n-u_1|^{r_1}|u_2^n-u_2|^{r_2} \, dx
= \int_{\R^N}|u_1|^{r_1}|u_2|^{r_2}\,dx+o(1).
$$
\end{lem}
\begin{proof}
{{This result is a direct consequence of \cite[Theorem 2]{BrLie} when using $j:\C \to \C$ defined by  $j(s+it):=|s|^{r_1}|t|^{r_2}$, for $s, t \in \R$, in that theorem. Here  $i$ is the imaginary unit.}}
\end{proof}

\begin{lem}\label{negative}
For any $b_1, b_2 \geq 0$ with $(b_1,b_2) \neq (0,0)$ if $(H_0)$ holds and $b_1 \neq 0$, $b_2 \neq 0$ if $(H_1)$ holds, we have for any $\rho >0$,
$$
\inf J(u_1,u_2) <0 \quad \mbox{on the set} \quad (u_1,u_2) \in S(b_1,b_2) \cap \mathcal{B}(\rho).
$$
\end{lem}
\begin{proof}
Consider for any ($u_1,u_2) \in S(b_1,b_2)$ the map introduced in \eqref{underscale}. {{Observing that $(\frac{p_i}{2}-1)N < 2,$ $i=1,2$ if $(H_0)$ holds and that $(\frac{r_1+r_2}{2}-1)N < 2$ if $(H_1)$ holds, then the lemma follows directly by letting $t \to 0^+$}}.
\end{proof}

Our next {result}, which is borrowed from N. Ikoma  \cite[Lemma 2.2]{Ik}, shows that when considering minimizing sequences to \eqref{i3} it is not restrictive to assume that the two components are {non-negative}.
\begin{lem} \label{nonnegative}
Assume that $\{(v_1^n, v_2^n)\}$ is a minimizing sequence to \eqref{i3}. If $\{(|v_1^n|, |v_2^n|)\}$ is compact in $H^1(\R^N) \times H^1(\R^N)$, so is $\{(v_1^n, v_2^n)\}$.
\end{lem}
\begin{proof}
First note that there exists $(w_1, w_2)$ such that, up to a subsequence, $(|v_1^n|, |v_2^n|) \rightarrow (w_1, w_2)$ in $H^1(\R^N) \times H^1(\R^N)$ and $(|v_1^n(x)|, |v_2^n(x)|) \rightarrow (w_1(x), w_2(x))$ for a.e. $x \in \R^N$. Since $\{(v_1^n, v_2^n)\}$ is a bounded sequence, there exists $(v_1, v_2) \in H^1(\R^N) \times H^1(\R^N)$ such that, up to a subsequence,
$(v_1^n, v_2^n) \rightharpoonup (v_1, v_2)$ in $H^1(\R^N) \times H^1(\R^N)$ and $(v_1^n(x), v_2^n(x)) \rightarrow (v_1(x), v_2(x))$ for a.e. $x \in \R^N$. By the uniqueness of the limit,  $w_i = |v_i|$ and then $(v_1^n, v_2^n) \rightarrow (v_1, v_2)$ in $L^2(\R^N) \times L^2(\R^N)$. Now since  $(v_1^n, v_2^n) \rightarrow (v_1, v_2)$ in $L^p(\R^N) \times L^p(\R^N)$ for $2 < p <2^*$ it follows that
$$
m(a_1, a_2)= J(v_1^n, v_2^n) + o(1) \geq J(v_1, v_2) \geq m(a_1, a_2),
$$
and thus $(v_1^n, v_2^n) \rightarrow (v_1, v_2)$ in $H^1(\R^N) \times H^1(\R^N)$.
\end{proof}

Next, recalling \eqref{underscale}, we define for $(u_1,u_2) \in H^1(\R^N)\times H^1(\R^N)$
\begin{align*}
\begin{split}
Q(u_1, u_2):&=\frac{d}{dt}J(u_1^t, u_2^t)|_{t=1} = \int_{\R^N} |\nabla u_1|^2 + |\nabla u_2|^2 \, dx\\ \nonumber
            &-\sum_{i =1}^2 \frac{\mu_i}{p_i} \left(\frac {p_i}{2} -1 \right)N \int_{\R^N}|u_i|^{p_i} \, dx
            -\beta  \left(\frac {r_1 + r_2}{2} -1 \right)N\int_{\R^N} |u_1|^{r_1}|u_2|^{r_2} \, dx.
\end{split}
\end{align*}
{Heuristically any solution $(u_1,u_2)$ of \eqref{sys1}, for some $(\lambda_1, \lambda_2) \in \R^2$, must satisfy the condition $Q(u_1,u_2) =0$. This can be proved rigorously by using the Pohozaev identity associated to \eqref{sys1}.
In particular the set defined by $Q =0$ corresponds to a natural constraint.}

\begin{lem} \label{la}
Assume $2 < p_1, p_2, r_1 + r_2 < 2^*$. If $(u_1, u_2) \neq (0, 0)$ solves the system \eqref{sys1} for some $(\lambda_1, \lambda_2) \in \R^2,$ then $\lambda_1 <0$ or $\lambda_2 <0$.
\end{lem}
\begin{proof}
Testing \eqref{sys1} by $(u_1, u_2)$ and integrating in $\R^N$, one  has
$$
\lambda_1 a_1 + \lambda_2 a_2 = \int_{\R^N}|\nabla u_1|^2 + |\nabla u_2|^2 \, dx
                              - \sum_{i =1}^2 \mu_i \int_{\R^N}|u_i|^{p_i} \, dx
                              {-}\beta(r_1 + r_2) \int_{\R^N}|u_1|^{r_1}|u_2|^{r_2} \,dx.
$$
Since $(u_1, u_2)$ satisfies \eqref{sys1}, then $Q(u_1, u_2) =0$, which implies
\begin{align*}
\lambda_1 a_1 + \lambda_2 a_2 &= \sum_{i =1}^2\left(\frac{\mu_i}{p_i}
                              \left(\frac {p_i}{2} -1\right)N - \mu_i\right)\int_{\R^N}|u_i|^{p_i}\,dx\\
                              & + \beta\left(\left(\frac{r_1 + r_2}{2} -1\right)N - (r_1 + r_2)\right)\int_{\R^N}|u_1|^{r_1}|u_2|^{r_2}\,dx < 0.
\end{align*}
\end{proof}
We recall that a sequence $\{(u_1^n, u_2^n)\} \subset S(a_1,a_2)$ is a Palais-Smale sequence for $J$ restricted to $S(a_1,a_2)$, at the level $c$, if
$J(u_1^n, u_2^n)  \to c $ and  $ (J_{\mid {S(a_1,a_2)}})'(u_1^n, u_2^n) \to 0.$

{The proof of our next lemma can be obtained by a direct adaptation of the one of \cite[Lemma 3.2]{BJ}. {Indeed observe that the convergence of $(u_1^n, u_2^n) \rightarrow (u_1, u_2)$ in $L^p(\R^N) \times L^p(\R^N)$ for $2 < p < 2^*,$ which is obtained at point (i) \cite[Lemma 3.2]{BJ} is not used to derive the points (i)-(iv) in that lemma.}

\begin{lem} \label{strongconv}
Assume $2 < p_1, p_2, r_1 + r_2 < 2^*$. For any bounded Palais-Smale sequence $\{(u_1^n, u_2^n)\}$ for $J$ restricted to $S(a_1,a_2)$, there exist  $(u_1,u_2) \in H^1(\R^N) \times H^1(\R^N)$, $(\la_1, \la_2) \in \R^2$ and a sequence $\{(\la_1^n, \la_2^n)\} \subset \R^2$ such that, up to a subsequence,
\begin{itemize}
\item[(i)]$(u_1^n, u_2^n) \rightharpoonup (u_1, u_2)$ in $H^1(\R^N) \times H^1(\R^N)$;
\item[(ii)] $(\la_1^n, \la_2^n)  \to (\la_1,\la_2)$ in  $\R^2$;
\item[(iii)] $J'(u_1^n, u_2^n) - \la_1^n (u_1^n,0) - \la_2^n (0, u_2^n) \to 0$ in $H^{-1}(\R^N) \times H^{-1}(\R^N)$;
\item[(iv)] $(u_1,u_2)$ is solution to the system \eqref{sys1} where $(\la_1,\la_2)$ {is given in (ii)}.
\end{itemize}
In addition, if $(u_1^n, u_2^n) \rightarrow (u_1, u_2)$ in $L^p(\R^N) \times L^p(\R^N)$ for $2 < p < 2^*,$ then $u_1^n \to u_1$ in $H^1(\R^N)$ if $\lambda_1 <0$. Similarly  $u_2^n\to u_2$ in $H^1(\R^N)$ if $\la_2<0$.
\end{lem}

\section{Proofs of Theorem \ref{thm1}(i) and Theorem \ref{thm2}(i)}\label{pro1}

\begin{lem} \label{mpgeo1}
Assume that $(H_0)$ or $(H_1)$ holds. There exist a $\beta_0=\beta_0(a_1, a_2) > 0$ and a $\rho_0= \rho_0(a_1, a_2) >0$, such that
{if $\, 0 \leq d_1 \leq a_1$, $ 0 \leq d_2 \leq a_2$ with $(d_1,d_2) \neq (0,0)$}, then
\begin{equation}\label{l32}
J(u_1,u_2) \geq 0  \quad \mbox{on} \quad S(d_1,d_2) \cap [\mathcal{B}(2 \rho_0) \backslash \mathcal{B}(\rho_0)]
\end{equation}
for any $0 < \beta \leq \beta_0$. In addition it holds that $\beta_0(a_1,a_2) \to \infty$ as $a_1 \to 0$ and $a_2 \to 0$.
\end{lem}
\begin{proof}
{Let us first consider the special case where $(d_1,d_2) = (a_1,a_2).$ For any $(u_1, u_2) \in S(a_1,a_2)$, let $\rho := \int_{\R^N} |\nabla u_1|^2 + |\nabla u_2|^2 dx.$ Then using } \eqref{ga1}-\eqref{ga2}, we have
\begin{align}\label{x}
\begin{split}
J(u_1, u_2) &= \frac12 \rho
                - \sum_{i=1}^2 \frac{\mu_i}{p_i}\int_{\R^N} |u_i|^{p_i} \,dx
                -\be\int_{\R^N}|u_1|^{r_1}|u_2|^{r_2} \, dx \\
            & \geq \frac12 \rho
                -\sum_{i = 1}^{2}
                K_i \left(\int_{\R^N} |\nabla u_i|^2 \,dx\right)^{\frac{N(p_i-2)}{4}} \\
            &-\beta K_3 \left(\int_{\R^N} |\nabla u_1|^2 \, dx \right)^{\frac{N(r_1q-2)}{4q}}
                \left(\int_{\R^N} |\nabla u_1|^2 \, dx \right)^{\frac{N(r_2q'-2)}{4q'}}\\
            &\geq \frac12 \rho
                - \sum_{i = 1}^{2} K_i \left(\int_{\R^N} |\nabla u_1|^2 + |\nabla u_2|^2 \,dx\right)^{\frac{N(p_i-2)}{4}} \\
            &-\beta K_3 \left(\int_{\R^N} |\nabla u_1|^2 + |\nabla u_2|^2 \, dx \right)^{\frac{N(r_1q-2)}{4q}}
                \left(\int_{\R^N} |\nabla u_1|^2 + |\nabla u_2|^2 \, dx \right)^{\frac{N(r_2q'-2)}{4q'}},\\
            & \geq \frac12 \rho - K_1 \rho^{\frac{N(p_1-2)}{4}} - K_2 \rho^{\frac{N(p_2-2)}{4}}
               -\beta  K_3  \rho^{\frac{N({r_1 + r_2}-2)}{4}}, \\
\end{split}
\end{align}
where
\begin{align}\label{ki}
K_i :=\frac{\mu_i}{p_i} {{C(N, p_i)}} a_i^{\frac{(1- {\alpha}(p_i))p_i}{2}} \quad \mbox{and} \quad  K_3 :={{C(N, r_1, r_2, q)}} a_1^{\frac{(1- {\alpha}(r_1q))r_1}{2}} a_2^{\frac{(1- {\alpha}(r_2q'))r_2}{2}}.
\end{align}

Now if $(H_0)$ holds, then $\frac{N(p_i-2)}{4}<1$ for $i=1,2$, and $ \frac{N(r_1+r_2 -2)}{4}>1$. We fix a  $\rho = \rho_0 >0$ sufficiently large so that
\begin{equation}\label{boundrho}
K_1 \rho_0^{\frac{N(p_1-2)}{4}-1} + K_2\rho_0^{\frac{N(p_2-2)}{4}-1} \leq  \frac18,
\end{equation}
and then we fix a $\beta_0 >0$ small enough, satisfying
\begin{equation}\label{boundbeta}
\beta_0 K_3 (2\rho_0)^{\frac{N({r_1 + r_2}-2)}{4}-1} \leq \frac18.
\end{equation}
Observe that the left hand side of \eqref{boundrho} and of \eqref{boundbeta} is a decreasing, respectively increasing, function of $\rho_0$ and thus we deduce that
\begin{equation}\label{locmin}
J(u_1,u_2) \geq \frac{1}{4}\rho_0 \quad \mbox{for }(u_1,u_2) \in \mathcal{B}(2 \rho_0) \backslash \mathcal{B}(\rho_0).
\end{equation}
If we assume that $(H_1)$ holds we fix a $\rho = \rho_0 >0$ sufficiently small so that
\begin{equation}\label{boundrhotilde}
K_1 (2\rho_0)^{\frac{N(p_1-2)}{4}-1} + K_2(2\rho_0)^{\frac{N(p_2-2)}{4}-1} \leq  \frac18,
\end{equation}
and then we fix a $\beta_0 >0$ small enough, satisfying
\begin{equation}\label{boundbetatilde}
\beta_0 K_3 \rho_0^{\frac{N({r_1 + r_2}-2)}{4}-1} \leq \frac18.
\end{equation}
Here again one can readily check that \eqref{locmin} holds.
{ Now to establish \eqref{l32} it suffices to observe that the choices of $\beta_0 >0$ and $\rho_0 >0$ done when $(d_1,d_2) = (a_1,a_2)$ are still valid in the general case.}
{ This follows directly from the observation that the $K_j$, $j=1,2,3$ are increasing functions of $a_1$ and $a_2$}. Finally we observe that, at the expense of requiring $a_1,a_2$ sufficiently small, it is possible to choose $\beta_0(a_1,a_2)$ arbitrarily large. Indeed, when $(H_0)$ holds, since $K_j \to 0$, $j=1,2,3$ as $a_i \to 0$, $i=1,2$, $\rho_0 >0$  in \eqref{boundrho} can be taken arbitrarily small and in \eqref{boundbeta}, $\beta_0 >0$ can be taken large if $\rho_0 >0$ is small. When $(H_1)$ holds we reach the same conclusion by similar arguments.
\end{proof}

From now on, for $a_1, a_2 \geq 0$ given, we fix a $\rho_0 >0$ and a $\beta_0 >0$ as determined in Lemma \ref{mpgeo1}.  For any $0 \leq d_1 \leq a_1$, $0 \leq d_2 \leq a_2$ we define
\begin{align}\label{another-min}
m(d_1, d_2):= \inf_{(u_1, u_2) \in S(d_1,d_2) \cap \mathcal{B}(\rho_0)} J(u_1, u_2).
\end{align}
\begin{lem}\label{another-sub-add}
Assume that $(H_0)$ or $(H_1)$ holds. Then for  $0 < \beta \leq  \beta_{0},$
\begin{itemize}
\item[(i)] {For any $0 \leq d_i  \leq a_i$, $i=1,2 $, $m(d_1,d_2) \leq 0$}. In addition if $(d_1,d_2) \neq (0,0)$ when $(H_0)$ hold or $d_1 \neq 0$ and $d_2 \neq 0$ when $(H_1)$ holds, we have $m(d_1,d_2) <0.$
\item[(ii)] If  $(d_1^n, d_2^n)$ is such that $(d_1^n, d_2^n) \to (d_1, d_2)$ as $n \to \infty$ with $0 \leq d_i^n \leq a_i$ for $i=1 ,2$, we have $m(d_1^n,d_2^n) \to m(d_1,d_2)$ as $n \to \infty$.
\item[(iii)]{For any $0 \leq d_i  \leq a_i$, $i=1,2 $,
$m(a_1, a_2) \leq m(d_1, d_2) + m(a_1-d_1, a_2-d_2).$}
\end{itemize}
\end{lem}

\begin{proof}
$(i)$ It follows directly from Lemma \ref{negative}. $(ii)$ By definition of $m(d_1^n,d_2^n)$,
there exists for any $\eps >0$,
$(u_1^n, u_2^n) \in S(d_1^n, d_2^n) \cap \mathcal {B}(\rho_0)$ such that $J(u_1^n, u_2^n) \leq m(d_1^n, d_2^n) + \eps.$
Setting $ \displaystyle {w_i^n := \frac{u_i^n}{\|u_i^n\|_2}d_i^{\frac12}}$ for $i =1, 2,$ we have $(w_1^n, w_2^n) \in S(d_1,d_2)$ and
$$
\|\nabla w_1^n\|_2^2 + \|\nabla w_2^n\|_2^2
= \|\nabla u_1^n\|_2^2 + \|\nabla u_2^n\|_2^2 + o(1)
< 2 \rho_0.
$$
{Consequently from the definition \eqref{another-min} we get}
\begin{align} \label{m1}
m(d_1, d_2)  \leq J(w_1^n, w_2^n)
                                = J(u_1^n, u_2^n) + o(1)
                                \leq m(d_1^n, d_2^n) + \eps +o(1),
\end{align}
and thus {\eqref{m1} gives that}
$
m(d_1, d_2) \leq m(d_1^n, d_2^n) + o(1).
$
Similarly, reversing the argument, it follows that
$
m(d_1^n, d_2^n) \leq m(d_1, d_2) + o(1).$
Now we deal with $(iii)$. { First we observe that, since for any $0 \leq c_i  \leq a_i$, $i=1,2 $, $m(c_1,c_2) \leq 0$, we can assume without restriction that  $m(d_1,d_2) + m(a_1-d_1,a_2-d_2) <0$. Now for } any $\eps >0$, there exist
$({\vphi}_1, {\vphi}_2) \in S(d_1,d_2) \cap \mathcal{B}(\rho_0)$ and $ (\psi_1, \psi_2) \in S(a_1-d_1,a_2-d_2)\cap \mathcal{B}(\rho_0)$
such that
\begin{equation} \label{ineq-j}
J({\vphi}_1, {\vphi}_2) \leq m(d_1, d_2) + \frac{\eps}{2} \quad \mbox{and} \quad
J(\psi_1, \psi_2) \leq m(a_1-d_1, a_2-d_2) + \frac{\eps}{2}.
\end{equation}
Setting
$
w_i = \{\vphi_i, \psi_i\}^{\star}$ for $ i = 1, 2$, it follows
 from, Lemma \ref{Ikoma} (iii)-(iv), that $(w_1, w_2) \in S(a_1,a_2)$ and
\begin{equation*}
\|\nabla w_1\|_2^2 + \|\nabla w_2\|_2^2
\leq \sum_{i = 1}^2 \|\nabla \vphi_i\|_2^2 + \|\nabla \psi_i\|_2^2.
\end{equation*}
If $\|\nabla w_1\|_2^2 + \|\nabla w_2\|_2^2 < \rho_0$, using {Lemmas \ref{Ikoma} and  \ref{Ikoma1} we have, from \eqref{ineq-j},}
\begin{align*}
m(a_1, a_2)  \leq J(w_1, w_2)
                     \leq J(\vphi_1, \vphi_2) + J(\psi_1, \psi_2)
                     \leq m(d_1, d_2)
                     + m (a_1-d_1, a_2-d_2) + {\eps},
\end{align*}
from which it follows that
$
m(a_1, a_2) \leq m (d_1, d_2) + m(a_1-d_1, a_2-d_2).
$
Otherwise $\rho_0 \leq \|\nabla w_1\|_2^2 + \|\nabla w_2\|_2^2 < 2\rho_0$ and in view of \eqref{l32}, we get
\begin{align*}
0 \leq  J(w_1, w_2)
\leq J(\vphi_1, \vphi_2)
+ J(\psi_1, \psi_2)\leq m (d_1, d_2)
+ m(a_1-d_1, a_2-d_2) + {\eps},
\end{align*}
which is impossible { since $m(d_1,d_2) + m(a_1-d_1,a_2-d_2) <0$.}
\end{proof}

\begin{lem}\label{conv1}
Assume that $(H_0)$ or $(H_1)$ holds. Any minimizing sequence to \eqref{i3} is, up to translation, compact in $L^p(\R^N)\times L^p(\R^N)$ for any $2<p<2^*$ as $0 < \beta \leq \beta_{0}$.
\end{lem}
\begin{proof}
The proof follows closely the one of \cite[Lemma 3.3]{GJ}.
Let
$\{(v_1^n, v_2^n)\}$ be a minimizing sequence to \eqref{i3}.
If
$$
\sup_{y\in \R^N}\int_{B(y, R)}|v_1^n|^2+|v_2^n|^2\, dx = o(1),
$$
for some $R>0$, then $v_i \rightarrow 0$ in $L^p(\R^N) $ for $2 < p < 2^*, i = 1,2$, see \cite [Lemma I. 1] {Li2}. This contradicts the property $m(a_1, a_2) < 0$, obtained in Lemma \ref{another-sub-add} (i).
Thus, there exist a $\gamma_0 >0$ and a sequence $\{y_n\} \subset \R^N$ such that
$$
\int_{B(y_n, R)}|v_1^n|^2+|v_2^n|^2\, dx \geq \gamma_0 > 0,
$$
and we deduce that $\left(v_1^n(x-y_n), v_2^n(x-y_n)\right) \rightharpoonup (v_1, v_2) \neq (0, 0)$ in $H^1(\R^N) \times H^1(\R^N)$.
Our aim is to prove that $w_i^n(x) := v_i^n(x) - v_i(x+y_n) \rightarrow 0$ in $L^p(\R^N)$ for $2 < p < 2^*, i = 1, 2$ and so we suppose by contradiction that there exists a $2<q<2^*$ such that $(w_1^n, w_2^n) \nrightarrow(0, 0)$ in $L^q(\R^N) \times L^q(\R^N).$ Still using \cite [Lemma I. 1] {Li2} it follows that there exists a sequence
$\{z_n\} \subset \R^N$ such that
$
\left(w_1^n(x-z_n), w_2^n(x-z_n)\right)\rightharpoonup (w_1, w_2)\neq(0, 0)
$
in $H^1(\R^n)\times H^1(\R^N).$

Now, combining Lemma \ref{Lieb}, the Brezis-Lieb Lemma and the translational invariance, we see that
\begin{equation}\label{eqj}
\begin{split}
&J(v_1^n, v_2^n) = J(v_1^n(x-y_n), v_2^n(x-y_n)) \\
&= J(v_1^n(x-y_n)-v_1, v_2^n(x-y_n) - v_2) + J(v_1, v_2) + o(1) \\
&= J(w_1^n(x-y_n), {w_2^n(x -y_n)} + J(v_1, v_2) + o(1) \\
&= J(w_1^n(x-z_n)-w_1, w_2^n(x-z_n)-w_2) + J(w_1, w_2) + J(v_1, v_2) + o(1),
\end{split}
\end{equation}
and
\begin{align*}
a_i = \|v_i^n(x-y_n)\|_2^2
                         &= \|v_i^n(x-y_n)-v_i\|_2^2 + \|v_i\|_2^2 + o(1)\\
                         &= \|w_i^n(x-z_n)-w_i+w_i\|_2^2 + \|v_i\|_2^2 + o(1)\\
                         &= \|w_i^n(x-z_n)-w_i\|_2^2 + \|w_i\|_2^2 + \|v_i\|_2^2 + o(1).
\end{align*}

Setting for $i=1,2$, $b_i := a_i - \|w_i\|_2^2 - \|v_i\|_2^2$ we then have $\|w_i^n(x-z_n)-w_i\|_2^2 = b_i + o(1).$ Thus
recording that $J(v_1^n,v_2^n) \to m(a_1,a_2)$, in view of \eqref{eqj} and Lemma \ref{another-sub-add} $(ii)$ we get
\beq[3.16]
m(a_1, a_2) \geq  J(w_1, w_2) + J(v_1, v_2) +  m(b_1, b_2).
\eeq
If $J(w_1, w_2) > m(\|w_1\|_2^2, \|w_2\|_2^2)$ or $J(v_1, v_2) > m(\|v_1\|_2^2, \|v_1\|_2^2),$ then,
from \eqref{eq:3.16} and Lemma \ref{another-sub-add} $(iii)$ , it follows
$$
m(a_1, a_2) >  m(\|w_1\|_2^2, \|w_2\|_2^2) + m(\|v_1\|_2^2, \|v_2\|_2^2) + m(b_1, b_2) \geq m(a_1 ,a_2)
$$
which is impossible. Hence $J(w_1, w_2) = m(\|w_1\|_2^2, \|w_2\|_2^2)$ and $J(v_1, v_2) = m(\|v_1\|_2^2, \|v_2\|_2^2).$

We denote by $\tilde{v}_i, \tilde{w}_i$ the classical Schwarz symmetric-decreasing rearrangement of $v_i, w_i$ for $i = 1, 2,$. Since
$$
 \|\tilde{v}_i\|_2^2 = \|v_i\|_2^2 ,\quad
\|\tilde{w}_i\|_2^2 = \|w_i\|_2^2,
$$
$$
J(\tilde{v}_1, \tilde{v}_2) \leq J(v_1, v_2), \quad
J(\tilde{w}_1, \tilde{w}_2) \leq J(w_1, w_2)
$$
see for example \cite{LL}, we deduce that
{$J(\tilde{v}_1, \tilde{v}_2) =   m(\|v_1\|_2^2, \|v_2\|_2^2)$ } and $J(\tilde{w}_1, \tilde{w}_2) = m(\|w_1\|_2^2, \|w_2\|_2^2).
$
Therefore, $(\tilde{v}_1, \tilde{v}_2), (\tilde{w}_1, \tilde{w}_2)$ are solutions to the system \eqref{sys1} for some $(\lambda_1, \lambda_2 ) \in \R^2$ and from standard regularity results we have that
$\tilde{v}_i, \tilde{w}_i \in C^2(\R^N)$ for $i = 1, 2.$

At this point Lemma \ref{Ikoma} comes into play. Without restriction, we may assume  $v_1 \neq 0$. We  distinguish two cases.

\noindent{\it Case 1: $v_1 \neq 0$ and $w_1 \neq 0$.} \\
By virtue of Lemma \ref{Ikoma} $ (ii), (iv), (v)$,
$$
\int_{\R^N} | {\nabla \{ \tilde{v}_1,  \tilde{w}_1\}}^{\star}| \, dx
         < \int_{\R^N} |\nabla \tilde{v}_1|^2 + |\nabla \tilde{w}_1|^2 \, dx
         \leq \int_{\R^N} |\nabla v_1|^2 + |\nabla w_1|^2 \, dx.
$$
{{Also from Lemma \ref{Ikoma1},
\begin{align}\label{zut}
\int_{\R^N}|\{\tilde{v}_1, \tilde{w}_1\}^{\star}|^{r_1}|\{\tilde{v}_2, \tilde{w}_2\}^{\star}|^{r_2} \, dx
        \geq \int_{\R^N} |v_1|^{r_1} |v_2|^{r_2}
              + |w_1|^{r_1} |w_2|^{r_2} \, dx,\nonumber
\end{align}
}
and thus }
\beq [re1]
J(v_1, v_2) + J(w_1, w_2) > J(\{\tilde{v}_1, \tilde{w}_1\}^{\star}, \{\tilde{v}_2, \tilde{w}_2\}^{\star}).
\eeq
Also from Lemma \ref{Ikoma} $(iii)$, for $i=1,2$,
\beq [re2]
{\int_{\R^N} |\{\tilde{v}_i, \tilde{w}_i\}^{\star}|^2 \, dx
= \int_{\R^N} |\tilde{v}_i|^2 + |\tilde{w}_i|^2 \, dx
= \int_{\R^N} |v_i|^2 + |w_i|^2 \, dx =a_i-b_i},
\eeq
and hence taking \eqref{eq:3.16}-\eqref{eq:re2} and Lemma \ref{another-sub-add} $(iii)$ into consideration, one obtains the contradiction
$$m(a_1, a_2) > m(b_1, b_2) + m(a_1-b_1, a_2-b_2) \geq m(a_1, a_2).$$

\noindent{\it Case 2: $v_1 \neq 0$, $w_1 =0$ and $w_2 \neq 0$.} \\
If $v_2 \neq 0$, we can reverse the role of $v_1, w_1$ and $v_2, w_2$ in {\it Case 1} to get a contradiction. Thus, we suppose that
$v_2 = 0.$

Then using  Lemma \ref{Ikoma} $(ii)$-$(v)$ { and Lemma \ref{Ikoma1} we get}
\beq [re3]
\begin{split}
J(\{\tilde{v}_1, 0\}^{\star}, {\{0, \tilde{w}_2 \}}^{\star})
&\leq  \frac12 \int_{\R^N} |\nabla \tilde{v}_1|^2 +  |\nabla \tilde{w}_2|^2 \, dx
        - \frac{\mu_1}{p_1} \int_{\R^N} |\tilde{v}_1|^{p_1} \, dx\\
&- \frac{\mu_2}{p_2} \int_{\R^N} |\tilde{w}_2|^{p_2} \, dx
        -\be \int_{\R^N} |\tilde{v}_1|^{r_1} |\tilde{w}_2|^{r_2} \, {dx}\\
& <J(\tilde{v}_1, 0) + J(0, \tilde{w}_2) \leq J(v_1, 0) + J(0, w_2),
\end{split}
\eeq
with

$ ||\{\tilde{v}_1, 0\}^{\star}||_2^2 = ||\tilde{v}_1||_2^2 = ||v_1||_2^2$  and  { $||\{0, \tilde{w}_2\}^{\star}||_2^2 = ||\tilde{w}_2||_2^2 = ||w_2||_2^2.$}
{Note that since $\tilde{v}_i \neq 0$ and $\tilde{w}_i \neq 0$ are the Schwarz symmetric decreasing rearrangement necessarily $\int_{\R^N} |\tilde{v}_1|^{r_1} |\tilde{w}_2|^{r_2} \,  dx >0$ and this guarantees the strict inequality in \eqref{eq:re3}. }
Thus using \eqref{eq:3.16}, \eqref{eq:re3} and Lemma \ref{another-sub-add}, we also have
$$m(a_1, a_2) > m(b_1, b_2) + m(a_1-b_1, a_2-b_2) \geq m(a_1, a_2).$$
The contradictions obtained in {\it Cases} 1 and 2 indicate that $w_i^n(x) = v_i^n(x) - v_i(x+y_n) \rightarrow 0$ in $L^p(\R^N)$ for $2 < p < 2^*, i = 1, 2.$
\end{proof}

\begin{proof}[Proof of Theorem \ref{thm1}$(i)$]

Let $\{(v_1^n, v_2^n)\} $ be an arbitrary minimizing sequence to \eqref{i3}. In view of  Lemma \ref{conv1}, there exists $(v_1, v_2)  \in H^1(\R^N) \times H^1(\R^N) $ such that, up to a subsequence, $(v_1^n, v_2^n) \rightharpoonup (v_1, v_2)$ in $H^1(\R^N) \times H^1(\R^N)$ and $(v_1^n, v_2^n) \rightarrow (v_1, v_2)$ in $L^p(\R^N) \times L^p(\R^N) $ for $2 < p < 2^*.$ Hence, by the weak lower semi-continuity of the norm,
$\|\nabla v_1\|_2^2 + \|\nabla v_2\|_2^2 < \rho_0$, namely $(v_1,v_2) \in \mathcal{B}(\rho_0)$, and $J(v_1, v_2) \leq m(a_1, a_2) <0$,
from which we deduce  that $(v_1, v_2) \neq (0, 0)$. To show the compactness of $\{(v_1^n, v_2^n)\}$ in $H^1(\R^N) \times H^1(\R^N)$, it suffices to prove that $(v_1, v_2) \in S(a_1,a_2)$. Assume by contradiction that  $\|v_1\|_2^2 := \bar{a}_1 < a_1$ or $\|v_2\|_2^2 := \bar{a}_2 < a_2.$  Then by the definition \eqref{another-min},
$m(\bar{a}_1, \bar{a}_2) \leq J(v_1, v_2)$. At this point, in light of Lemma \ref{another-sub-add} (i) and (iii), we get
$$
J(v_1, v_2) \leq m(a_1, a_2) \leq m(\bar{a}_1,\bar{a}_2) + m(a_1 - \bar{a}_1, a_2 - \bar{a}_2) < m(\bar{a}_1,\bar{a}_2) \leq J(v_1,v_2).
$$
This contradiction proves that $(v_1, v_2) \in S(a_1,a_2)$. To end the proof we note that without restriction we can choose a minimizer $(v_1, v_2)$ of $m(a_1,a_2)$ with $v_1 \geq 0 $ and $v_2 \geq 0$. {Note that $(v_1, v_2)$ is a solution to \eqref{sys1} for some $(\lambda_1, \lambda_2) \in \R^2$, where the parameters $\lambda_1, \lambda_2$ are determined by the Lagrange's multiplier rule. Now by the elliptic regularity theory, we know that $v_1, v_2 \in C^2(\R^N)$. Since $\beta,\mu_1 >0$, then $-\Delta v_1 \geq \lambda_1 v_1$, hence using the maximum principle \cite[Theorem 2.10]{HL} we deduce that $v_1>0$. Similarly, we can obtain that $v_2>0$. Thus the proof is complete.}
\end{proof}

\begin{proof}[Proof of Theorem \ref{thm2}$(i)$]
Let us show that, up to translations, every minimizing sequence $\{(v_1^n, v_2^n)\} $ to  \eqref{i3} is {convergent}. To this aim we first deduce from Lemma \ref{nonnegative} that it is not restrictive to assume that the two components are non-negative. {Also from \cite[Corollary 1.3]{Gh} we see that there exists another minimizing sequence $\{(\tilde{v}_1^n, \tilde{v}_2^n)\} \subset S(a_1,a_2)$ which is a Palais-Smale sequence for $J$ restricted to $S(a_1, a_2)$, and such that $||(\tilde{v}_1^n, \tilde{v}_2^n) - (v_1^n, v_2^n)|| \to 0$ in $H^1(\R^N) \times H^1(\R^N)$}. Because of this convergence, we have in particular that $(\tilde{v}_1^n)^- \to 0$ and $(\tilde{v}_2^n)^- \to 0$ as $n \to \infty$ and we obtain that $(\tilde{v}_1^n, \tilde{v}_2^n) \rightharpoonup (v_1,v_2)$ in $H^1(\R^N) \times H^1(\R^N)$ with $v_1 \geq 0$ and $v_2 \geq 0$. Furthermore, it results from Lemma \ref{strongconv} that $(v_1, v_2)$ satisfies \eqref{sys1}-\eqref{sys2} with some $(\lambda_1, \lambda_2) \in \R^2$, from which we infer that $Q(v_1, v_2)=0$. {From Lemma \ref{conv1}, we get that, {up to translations,} $(\tilde{v}_1^n, \tilde{v}_2^n) \to (v_1,v_2)$ in $L^p(\R^N) \times L^p(\R^N)$ for $2 <p<2^*$, then by the weak semicontinuity of the norm, we obtain that $J(v_1,v_2) \leq m(a_1,a_2) <0$}. It remains to show that $(v_1,v_2) \in S(a_1,a_2)$. By Lemma \ref{la} we can assume without restriction that $\lambda_1 <0$ and then Lemma  \ref{strongconv} gives $v_1 \in S(a_1)$. If $\la_2 <0$ we also have that  $v_2 \in S(a_2).$ Let us thus assume by contradiction that  $\la_2 \geq 0$. In the case $1 \leq N \leq 4$, since
$$
- \Delta v_2 = \lambda_2 v_2 + \mu_2 v_2^{p_2-1} + \beta r_2 v_1^{r_1}v_2^{r_2 -1} \geq 0,
$$
by {the Liouville's result} recalled in Lemma \ref{liouville} we obtain that $v_2=0$. It then follows that $J(v_1, v_2)= J(v_1, 0)$ with $v_1 \in S(a_1)$ satisfying $- \Delta v_1 = \lambda_1 v_1 + \mu_1 v_1^{p_1 -1}$.
{Thus, since $p_1 >2 + \frac 4N$, we deduce from \cite[Lemma 3.1]{BJ} that $J(v_1,0) >0$ and this provides the contradiction. }
If we now assume that
$N \geq 5$, testing the second equation of (\ref{sys1}) with $v_2$, and integrating in $\R^N$, we get, because $\lambda_2 \geq 0$, that
\begin{equation}\label{test-lagrange}
\int_{\R^N} |\nabla v_2|^2 \, dx - \mu_2 \int_{\R^N} |v_2|^{p_2} \, dx - \beta r_2 \int_{\R^N} |v_1|^{r_1} |v_2|^{r_2} \, dx \geq 0.
\end{equation}
Now jointing (\ref{test-lagrange}) with $Q(v_1,v_2)=0$, we obtain that
\begin{align} \label{v}
\begin{split}
0 & \geq  \int_{\R^N} |\nabla v_1|^2  - \frac{\mu_1}{p_1} \left( \frac{p_1}{2}-1\right)N  \int_{\R^N} |v_1|^{p_2} \, dx
 + \left( \mu_2 - \frac{\mu_2}{p_2}\left(\frac{p_2}{2} -1\right) N \right) \int_{\R^N} |v_2|^{p_2} \, dx\\
&  + \beta \left( r_2 - \left(\frac{r_1 + r_2}{2} -1\right) N \right) \int_{\R^N} |v_1|^{r_1} |v_2|^{r_2} \, dx.
\end{split}
\end{align}
Note that the coefficient of $\int_{\R^N} |v_2|^{p_2} \, dx$ is positive. {Thanks to $v_1 \neq 0$}, from the Gagliardo-Nirenberg inequality \eqref{ga1}, we can assume without restriction that
\begin{align}\label{con}
{\int_{\R^N} |\nabla v_1|^2  - \frac{\mu_1}{p_1}\Big(\frac{p_1}{2}-1\Big)N  \int_{\R^N} |v_1|^{p_1} \, dx > 0}
\end{align}
by taking, if necessary, $\rho_0 >0$ (and thus $\beta_0 >0$) smaller in Lemma \ref{mpgeo1}. {Thus by taking into consideration \eqref{con}, we obtain a contradiction from \eqref{v}}, since we have assumed that ${r_2 \geq  \left(\frac{r_1 + r_2}{2} -1\right)N}$. Knowing that $\lambda_2 <0$, we deduce that $v_2 \in S(a_2)$ and then we conclude as before that $v_1 >0$ and $v_2 >0$.
\end{proof}

\section{Proofs of Theorem \ref{thm1}(ii) and Theorem \ref{thm2}(ii)}\label{pro2}

To obtain our second solution and in order to benefit from additional compactness we replace $H^1(\R^N) \times H^1(\R^N) $ by $H^1_{rad}(\R^N) \times H^1_{rad}(\R^N)$.
It is well-known that the subspace of $H^1(\R^N)$ consisting of radially symmetric functions
$H^1_{rad}(\R^N)$ is compactly embedded into $L^q(\R^N)$ for $2 <q < 2^*$ and $N \geq 2$. Also it is classical that a constrained critical point of $J$ defined on $H^1_{rad}(\R^N) \times H^1_{rad}(\R^N)$ is a constrained critical point of $J$ defined on $H^1(\R^N) \times H^1(\R^N)$. Accordingly we define $S_{rad}(a_1,a_2) := S(a_1,a_2) \cap {H^1_{rad}(\R^N) \times H^1_{rad}(\R^N)}$.
\vspace{2mm}

 First we deal with the case where $(H_0)$ holds. By Lemmas \ref{negative} and \ref{mpgeo1} we know that there exists a
 $0 < \bar{\rho} = \bar{\rho}(a_1,a_2)  < \rho_0$ such that, for any $0 < \beta \leq \beta_0$,
\begin{align*}
\gamma(a_1, a_2): = \inf_{g \in \Gamma} \max_{t \in [0, 1]} J(g(t))> \max \{J(g(0)), J(g(1))\},
\end{align*}
where
\begin{align*}
\Gamma := \{g \in C([0, 1], S(a_1,a_2)): g(0) \in B(\bar{\rho}) , g(1) \notin \overline{B(\rho_0)} \mbox{ with } J(g(1)) < 0\}.
\end{align*}

\begin{lem}\label{lem121}
Assume that $(H_0)$ holds. Then, for any $0 < \beta \leq \beta_0$, there exists a Palais-Smale sequence $\{(u_1^n, u_2^n)\} \subset S(a_1,a_2)$ for $J$ restricted to $S_{rad}(a_1,a_2)$ at the level $\gamma(a_1,a_2)$ which satisfies  $(u_1^n)^- \to 0$, $(u_2^n)^- \to 0$ in $H^1(\R^N)$ and the property  $Q(u_1^n, u_2^n) \to 0$.
\end{lem}

\begin{proof}
The proof of {such a result} is now standard, similar statements appear in \cite{Jeanjean,BJ,BJS}. For a closely related version we refer to \cite[Lemma 5.5]{BJ}.
\end{proof}

\begin{lem}\label{lem122}
Assume that $(H_0)$ holds and that $0 < \beta \leq \beta_0$. Then there exists $(u_1, u_2) \in H_{rad}^1(\R^N)\times H_{rad}^1(\R^N)$ solving \eqref{sys1} for some $(\lambda_1,\lambda_2) \in \R^2$ such that $J(u_1,u_2) = \gamma(a_1,a_2)$.
\end{lem}

\begin{proof}
The couple $(u_1,u_2)$ will be obtained as a weak limit of the Palais-Smale sequence whose existence is provided by Lemma \ref{lem121}. In this aim we first show that $\{(u_1^n,u_2^n)\}$ is bounded in $H^1(\R^N)\times H^1(\R^N)$. As we shall see this property follows from the fact that the functional $J$ restricted to the set where $Q=0$ is coercive. Indeed we can write, for any $\varepsilon >0$,
$$
\begin{aligned}
J(u_1,u_2)
 &= \frac{\varepsilon}{2} ||\nabla u_1^n||_2^2 + \frac{\varepsilon}{2} ||\nabla u_2^n||_2^2
     +  a_1(\varepsilon) ||u_1^n||_{p_1}^{p_1} + a_2(\varepsilon) ||u_2||_{p_2}^{p_2} \\
 &\hspace{1cm}
     + \be b(\varepsilon) \int_{\R^N}|u_1^n|^{r_1}|u_2^n|^{r_2} \, dx
     + \frac{1-\varepsilon}{2}Q(u_1^n,u_2^n).
\end{aligned}
$$
where
$$
a_1(\varepsilon) = \frac{(1 - \varepsilon)\mu_1N}{2 p_1}\left(\frac{p_1}{2}-1\right) - \frac{\mu_1}{p_1}, \quad a_2(\varepsilon) = \frac{(1 - \varepsilon)\mu_2N}{2 p_2} \left(\frac{p_2}{2}-1\right) - \frac{\mu_2}{p_2}
$$
and
$$
b(\varepsilon) = \frac{(1 - \varepsilon)N}{2} \left(\frac{r_1 + r_2}{2} -1\right)-1.
$$
The coefficients $a_i(\varepsilon)$, $i=1,2$ are strictly negative but the corresponding terms can be controlled by $\varepsilon||\nabla u_i^n||_2^2$ using the Gagliardo-Nirenberg inequality \eqref{ga1} because $p_i<2+\frac4N$ for $i=1,2$. Now since $r_1 + r_2 > 2 + \frac{4}{N}$ we also have that $ b(\varepsilon) >0$ for $\varepsilon >0$ small enough. Recalling that  $Q(u_1^n,u_2^n) \to 0$ the boundedness of our Palais-Smale sequence follows.

At this point using Lemma \ref{strongconv} we can assume that $u_i^n \rightharpoonup u_i$, $i=1,2$ in $H^1(\R^N)$ and that $u_i^n \to u_i$, $i=1,2$ in $L^q(\R^N)$ with $q \in ]2,2^*[.$ Lemma \ref{strongconv} also insures that $(u_1,u_2)$ is a solution of \eqref{sys1} for some $(\lambda_1, \lambda_2) \in \R^2$ and thus $Q(u_1,u_2)=0$.  Clearly the property $u_1 \geq 0$ and $u_2 \geq 0$ follows from $(u_1^n)^- \to 0$, $(u_2^n)^- \to 0$ in $H^1(\R^N)$. {Arguing as the proof of Theorem \ref{thm1} $(i)$, from the maximum principle \cite[Theorem 3.10]{HL} we get that $u_1, u_2 >0$}. It remains to show that $J(u_1,u_2) = \gamma(a_1,a_2)$. Since $Q(u_1^n,u_2^n) \to 0$ we have that
{
\begin{align}\label{Q1}
\begin{split}
\int_{\R^N} |\nabla u_1^n|^2 + |\nabla u_2^n|^2 \, dx & =
            \sum_{i =1}^2 \frac{\mu_i}{p_i} \left(\frac {p_i}{2} -1 \right)N \int_{\R^N}|u_i^n|^{p_i} \, dx \\
            &+\beta  \left(\frac {r_1 + r_2}{2} -1 \right)N\int_{\R^N} |u_1^n|^{r_1}|u_2^n|^{r_2} \, dx +o(1).
\end{split}
\end{align}
}
From the strong convergence in $L^q(\R^N)$ for $q \in ]2, 2^*[$,  the right hand side of \eqref{Q1} converges to
$$
\sum_{i=1}^2 \frac{\mu_i}{p_i} \left(\frac{p_i}{2} -1 \right)N \int_{\R^N}|u_i|^{p_i} \, dx
            + \beta  \left(\frac{r_1 + r_2}{2} -1 \right)N \int_{\R^N} |u_1|^{r_1}|u_2|^{r_2} \, dx.
$$
Thanks to $Q(u_1,u_2)=0$,
this gives that
$ \int_{\R^N} |\nabla u_1^n|^2 + |\nabla u_2^n|^2 \, dx \to \int_{\R^N} |\nabla u_1|^2 + |\nabla u_2|^2 \, dx$.
As a consequence, we deduce that $J(u_1^n, u_2^n) \to J(u_1,u_2)$. Thus recalling that $J(u_1^n,u_2^n) \to \gamma(a_1,a_2)$ we get $J(u_1,u_2) = \gamma(a_1,a_2).$
\end{proof}

\begin{proof}[Proof of Theorem \ref{thm1}$(ii)$] First we consider the case $2 \leq N \leq 4.$ In view of Lemma \ref{lem122}, it remains to prove that $(u_1,u_2) \in S(a_1,a_2)$. Recall that here we work in the radially symmetry space $H^1_{rad}(\R^N) \times H^1_{rad}(\R^N),$ thus in view of Lemma \ref{strongconv}, we only need to prove that $\lambda_1, \lambda_2 <0$. At this point, as in the proof of Theorem \ref{thm2}(i), reasoning by contradiction if necessary we assume that $\lambda_2 \geq 0$, we obtain that $J(u_1,u_2) = J(u_1,0)$ with  $u_1 \in S(a_1)$ solution to $- \Delta u_1 = \lambda_1 u_1 + \mu_1 u_1^{p_1 -1}$. Since  $p_1  < 2 + \frac 4N$, we necessarily have that $J(u_1,0) <0$, this provides the contradiction $J(u_1, 0)=\gamma(a_1, a_2)>0$. We then conclude as before.
\end{proof}

Let us now consider the case $N \geq 5$ where the Liouville's argument cannot be applied.

\begin{lem}\label{lem123}
Assume that $(H_0)$ {holds} and that either $p_i \leq r_1 + r_2 - \frac{2}{N}$, $i=1,2$ or $|p_1-p_2| \leq \frac{2}{N}$. { If $Q(v_1, v_2)=0$ and $J(v_1, v_2) >0$ }, then
\begin{equation}\label{keyp}
J(v_1,v_2) = \max_{t>0}J(v_1^t,v_2^t).
\end{equation}
\end{lem}

The proof of Lemma \ref{lem123} relies on the following technical result whose proof will be postponed until the Appendix.

\begin{lem}\label{lem124}
Assume that $(H_0)$ holds and that either $p_i \leq r_1 + r_2 - \frac{2}{N}$, $i=1,2$ or $|p_1-p_2| \leq \frac{2}{N}$. { Let $(v_1, v_2) \in H^1(\R^N) \times H^1(\R^N)$ be arbitrary. Then the function $t \mapsto J(v_1^t, v_2^t)$ } admits at most {two stationary points} for $t>0$.
\end{lem}

\begin{remark}\label{extension}
It is only in the proof of Lemma \ref{lem124} that we need the assumption $p_i \leq r_1 + r_2 - \frac{2}{N},$ $i =1,2$ or alternatively $|p_2-p_1| \leq \frac{2}{N}$. These conditions are used to establish the key property on which our proof of Theorem \ref{thm1}(ii) relies, {namely that the couple $(u_1, u_2) \in H^1(\R^N)\times H^1(\R^N)$ obtained in Lemma \ref{lem122}  satisfies $J(u_1,u_2) = \max_{t >0}J(u_1^t,u_2^t)$. }
\end{remark}

\begin{proof}[Proof of Lemma \ref{lem123}]
{{
First observe that since $Q(v_1,v_2) = 0$ and $J(v_1,v_2) >0$, necessarily
\begin{equation}\label{L1}
\int_{\R^N} |v_1|^{r_1}|v_2|^{r_2} \, dx >0
\end{equation}
and thus in particular $v_1 \neq 0$ and $v_2 \neq 0$. To check \eqref{L1}, assume by contradiction that the integral is zero. Then $J(v_1, v_2 ) >0$ gives
\begin{equation}\label{L2}
\frac{1}{2}\int_{\R^N} |\nabla v_1|^2 + |\nabla v_2|^2 \, dx > \sum_{i =1}^2 \frac{\mu_i}{p_i} \int_{\R^N}|v_i|^{p_i} \, dx
\end{equation}
and using that $Q(u_1,u_2) =0$ namely that
\begin{equation*}
\int_{\R^N} |\nabla v_1|^2 + |\nabla v_2|^2 \, dx = \sum_{i =1}^2 \frac{\mu_i}{p_i} \left(\frac {p_i}{2} -1 \right)N \int_{\R^N}|v_i|^{p_i} \, dx
\end{equation*}
we get, since $p_i < 2 + \frac{4}{N}, i=1,2$ a contradiction with \eqref{L2}}}.

Next we deduce from Lemma \ref{negative} that $h(t): = J(u_1^t,u_2^t)$ must have a local minimizer for a $t_0 \in ]0,1[$ with $h(t_0)<0$. Also since \eqref{L1} holds and $r_1+r_2> 2+ \frac {4}{N}$, then $h(t) \to -\infty$ as $t \to \infty$. At this point we deduce, using Lemma \ref{lem124}, that necessarily $h$ admits a unique second stationary point. Since $Q(v_1,v_2)=0$ it follows by identification that  $\max_{t>0}h(t)=h(1)$ and \eqref{keyp} holds.
\end{proof}

\begin{proof}[End of the proof of Theorem \ref{thm1}$(ii)$.] We now deal with the case $N \geq 5$. In view of Lemma \ref{lem122}, it remains to prove that $(u_1,u_2) \in S(a_1,a_2)$. Let $\bar{a}_1:= \|u_1\|_2^2 \leq  a_1$ and $\bar{a}_2:= \|u_2\|_2^2 \leq  a_2$. Assuming by contradiction that either $\bar{a}_1 < a_1$ or $\bar{a}_2 < a_2$ we shall obtain a contradiction by constructing a path $g \in \Gamma$ such that
$$\max_{t \in [0,1]}J(g(t)) < \gamma(a_1,a_2).$$
Let $0 < t_1 < 1 < t_2$ be such that $(u_1^{t_1},u_2^{t_1}) \in {\mathcal{B}(\bar{\rho}/2)}$ and $J(u_1^{t_2}, u_2^{t_2}) < m(a_1,a_2) <0.$ The existence of $0< t_1 < 1$ is insured by Lemma \ref{negative} and the one of $t_2 >1$ by the property that $J(u_1^t, u_2^t) \to - \infty$ as $t \to \infty$. Now because of \eqref{i4}, if $\bar{a}_1 < a_1$ there exists a $w_1 \in S(a_1 - \bar{a}_1)$ such that $w_1^{t_1} \in \mathcal{B}(\bar{\rho}/2)$, and $J(w_1^t,0) <0$ for $t \in [t_1, t_2]$. Here $w^t(x):= t^{\frac{N}{2}} w(tx)$ and without restriction we can assume that $w_1 \in S(a_1 - \bar{a}_1)$ is radially symmetric.   Similarly if $\bar{a}_2 < a_2$ we can choose a radially symmetric $w_2 \in S(a_2 - \bar{a}_2)$ such that $w_2^{t_1} \in \mathcal{B}(\bar{\rho}/2)$, and $J(0, w_2^{t}) < 0$ for $t \in [t_1, t_2]$. Note that we just take $w_1=0$ if $\bar{a}_1 = a_1$, and $w_2=0$ if  $\bar{a}_2 = a_2$.

We now set
$$
v_i := \{u_i, w_i\}^*, \ \ \mbox{for} \ \ i=1,2,
$$
where $\{u, v\}^*$ is the rearrangement of $u, v$ defined by \eqref{rearr}.
Then we consider the path  $ [t_1, t_2] \mapsto (v_1^t, v_2^t)$.
From Lemma \ref{Ikoma} $(iii)$-$(iv)$, for all $t \in [t_1, t_2]$, we see that $(v_1^t, v_2^t) \in S(a_1, a_2)$, and
\begin{align*}
{||\nabla v_1^t||_2^2 + ||\nabla v_2^t||_2^2 =} & \, {t^2 \left(||\nabla v_1||_2^2 + ||\nabla v_2||_2^2\right)} \leq t^2 \sum_{i=1}^2 {\left(||\nabla u_i||_2^2 + ||\nabla w_i||_2^2\right)} \\
&  = \sum_{i=1}^2 {\left( ||\nabla u_i^t||_2^2 + ||\nabla w_i^t||_2^2\right)}.
\end{align*}
Thus $(v_1^{t_1},v_2^{t_2}) \in \mathcal{B}(\bar{\rho})$, due to $(u_1^{t_1},u_2^{t_1}), (w_1^{t_1},w_2^{t_1}) \in \mathcal{B}(\bar{\rho}/2)$. Also
\begin{align*}
J(v_1^t, v_2^t) = & \frac{t^2}{2} \int_{\R^N} |\nabla v_1|^2 + |\nabla v_2|^2 dx - \sum_{i=1}^2 \frac{\mu_i}{p_i} t^{(\frac{p_i}{2}-1)N} \int_{\R^N} |v_i|^{p_i} dx\\
              & - \beta t^{(\frac{r_1+r_2}{2}-1)N} \int_{\R^N} |v_1|^{r_1}|v_2|^{r_2}\, dx\\
              & \leq  \frac{t^2}{2} \int_{\R^N} |\nabla u_1|^2 + |\nabla u_2|^2 dx + \frac{t^2}{2} \int_{\R^N} |\nabla w_1|^2 + |\nabla w_2|^2\, dx \\
              & - \sum_{i=1}^2 \frac{\mu_i}{p_i} t^{(\frac{p_i}{2}-1)N} \int_{\R^N} |u_i|^{p_i}\, dx - \sum_{i=1}^2 \frac{\mu_i}{p_i} t^{(\frac{p_i}{2}-1)N} \int_{\R^N} |w_i|^{p_i}\, dx \\
              & - \beta t^{(\frac{r_1+r_2}{2}-1)N} \int_{\R^N} |u_1|^{r_1}|u_2|^{r_2} dx,
\end{align*}
{where we have used Lemma \ref{Ikoma1}}.
As a consequence, for $ t \in [t_1, t_2]$,
\begin{align} \label{decrease}
J(v_1^t,v_2^t) \leq J(u_1^t, u_2^t) +  J(w_1^t, 0) + J(0, w_2^t).
\end{align}
{Since  $ J(w_1^{t_2}, 0) \leq 0 $  and $J(0, w_2^{t_2})\leq 0$, we get from \eqref{decrease} that  $J(v_1^{t_2},v_2^{t_2}) \leq J(u_1^{t_2},u_2^{t_2}) < m(a_1,a_2) $. In particular, from the definition, see \eqref{i3}, of $m(a_1,a_2)$ this shows that $(v_1^{t_2},v_2^{t_2}) \notin \overline{\mathcal{B}(\rho_0)}$}.  {Now still using \eqref{decrease} and since, see Lemma \ref{lem122}, $J(u_1,u_2) = \gamma(a_1,a_2)$ we deduce from Lemma  \ref{lem123} that }
\begin{align*}
\max_{t \in [t_1,t_2]} J(v_1^t,v_2^t)  \leq &  \max_{t \in [t_1,t_2]} \Big[J(u_1^t, u_2^t) + J(w_1^t,0) + J(0, w_2^t)\Big]\\
              & \leq \max_{t \in [t_1,t_2]} J(u_1^t, u_2^t) + \max_{t \in [t_1,t_2]} J(w_1^t, 0)+ \max_{t \in [t_1,t_2]} J(0, w_2^t) \\
              & \leq  J(u_1,u_2) + \max_{t \in [t_1,t_2]} J(w_1^t, 0)+ \max_{t \in [t_1,t_2]} J(0, w_2^t) < \gamma(a_1,a_2),
\end{align*}
because $\max_{t \in [t_1,t_2]} J(w_1^t, 0) <0$ if $w_1 \neq 0$ and $\max_{t \in [t_1,t_2]} J(0, w_2^t) <0$ if $w_2 \neq 0$. Thus, after a renormalization $[t_1,t_2] \to [0,1]$ we obtain a path $g$ lying in $\Gamma$ such that
$\max_{t \in [0,1]}J(g(t)) < \gamma(a_1,a_2)$ and this ends the proof.
\end{proof}

We now turn to the existence of the second solution of Theorem \ref{thm2}. Our proof { borrows} several key ingredients from \cite{BJS}. First we recall some properties of the scalar nonlinear Schr\"odinger equation. Let $w_{a, \mu, p} >0, w_{a, \mu, p} \in S(a)$ { satisfy}
\begin{align}\label{eqw}
-\Delta w_{a, \mu, p} - \lambda w_{a, \mu, p} = \mu |w_{a, \mu, p}|^{p-2}w_{a, \mu, p},
\end{align}
for $2 + \frac 4N < p < 2^*$ and $\lambda <0$. It is well known that $w_{a, \mu, p}$ is unique and given by
\begin{align}\label{solw}
w_{a, \mu, p}(x) = \left(- \frac{\lambda}{\mu}\right)^{\frac{1}{p-2}}w_0((-\lambda)^{\frac 12}x),
\end{align}
where $w_0$ is the unique positive radial solution of the equation
$
-\Delta w + w = |w|^{p-2}w.$
In what follows, we set
\begin{align} \label{defw}
C_0(N, p)= \int_{\R^N} |\nabla w_0|^2 \,dx , \quad \mbox{and} \quad C_1(N, p)= \int_{\R^N} |w_0|^p \,\,dx.
\end{align}
Let us now introduce the Pohozaev type manifold
\begin{align*}
\mathcal{P}(N, a, \mu, p):= \{u \in S(a): \int_{\R^N}|\nabla u|^2 \,dx = \frac{\mu}{p}\left(\frac{p}{2} - 1\right)N \int_{\R^N}|u|^p \,dx\}
\end{align*}
and the functional $I_{\mu, p} : H^1(\R^N) \to \R$ given by
$$I_{\mu, p}(u) = \frac{1}{2}\int_{\R^N} |\nabla u|^2 - \frac{\mu}{p} \int_{\R^N} |u|^p \, dx.$$

\begin{lem}
The solution $w_{a, \mu, p}$ of \eqref{eqw} belongs to  $\mathcal{P}(N, a, \mu, p)$,
and it minimizes the functional $I_{\mu, p}$ on the manifold $\mathcal{P}(N, a, \mu, p)$.
\end{lem}
\begin{proof}
The proof of such results can be directly deduced from  \cite[Lemmas 2.7 and 2.10]{Jeanjean}.
\end{proof}
From \eqref{solw}-\eqref{defw}, {and by following the arguments of \cite[Proposition 2.2]{BJS} } it is not difficult to check that
\begin{align} \label{normw}
\begin{split}
&\|\nabla w_{a, \mu, p}\|_2^2 = \left(\frac {a}{C_0(N, p)}\right)^{\frac{2p-N(p-2)}{4 - N(p-2)}} \mu^{\frac{4}{4-N(p-2)}} C_0(N, p),\\
& \|w_{a, \mu, p}\|_p^p = \left(\frac {a}{C_0(N, p)}\right)^{\frac{2p-N(p-2)}{4 - N(p-2)}} \mu^{\frac{N(p-2)}{4-N(p-2)}}C_1(N, p),
\end{split}
\end{align}
and then the least energy level of $I_{\mu, p}$ on $\mathcal{P}(N, a, \mu, p)$ is given by
\begin{align}\label{level}
l(N, a, \mu, p)&:= \inf_{u \in \mathcal{P}(N, a, \mu, p)}I_{\mu, p}(u) = I_{\mu, p}(w_{a, \mu, p}) = \frac{\mu}{p}\left(\left(\frac p2 -1\right)\frac N2 -1\right) \int_{\R^N} |w_{a, \mu, p}|^p \,dx  \nonumber \\
            & = \frac 1p \left(\left(\frac p2 -1\right)\frac N2 -1\right)\left(\frac {a}{C_0(N, p)}\right)^{\frac{2p-N(p-2)}{4 - N(p-2)}} \mu^{\frac{4}{4-N(p-2)}}C_1(N, p).
\end{align}
We now define, for $s \in \R$ and $w \in H^1(\R^N)$, the {dilation} $(s * w)(x) := e^{\frac{Ns}{2}}w(e^s x)$.

\begin{lem} \label{pro-ground}
For any $w \in H^1(\R^N)$, there holds
\begin{align} \label{psi1}
&I_{\mu, p}(s*w) = \frac{e^{2s}}{2} \int_{\R^N} |\nabla w|^2 \,dx
                    - \frac{\mu}{p} e^{s(\frac p2 - 1)N}\int_{\R^N} |w|^p \, dx, \\ \nonumber
& \frac{\partial}{\partial s} I_{\mu, p}(s*w)= e^{2s}\int_{\R^N}|\nabla w|^2 \,dx
                    -\frac{\mu}{p} \left(\frac p2 -1\right)N e^{s(\frac p2 - 1)N} \int_{\R^N} |w|^p \, dx.
\end{align}
In particular, if $ w = w_{a, \mu, p}$, then
\begin{align}\label{psi}
\begin{split}
&\frac{\partial}{\partial s} I_{\mu, p}(s*w_{a, \mu, p}) = 0 \,\,\text{if}\,\, s=0,\\
&\frac{\partial}{\partial s} I_{\mu, p}(s*w_{a, \mu, p}) > 0(< 0)  \,\,\text{if}\,\, s < 0(>0).
\end{split}
\end{align}
\end{lem}
\begin{proof}
We refer to \cite[Lemma 3.1]{BJS} for a very similar proof.
\end{proof}
Now define, for $i=1, 2,$
\begin{align} \label{defc}
\begin{split}
c_i:=c_i(r_1 + r_2, p_i)&:=\frac{p_i-(r_1+r_2)}{p_i}\left(\frac{p_i(r_1+r_2)}{p_i-2}\right)^{\frac{r_1+r_2-2}{p_i-(r_1+r_2)}}\\
                        &=\max_{t \geq 0} \Big[t^{r_1+r_2-2}-\frac {1}{p_i}t^{p_i-2}\Big].
\end{split}
\end{align}
In view of \eqref{level}, {$l(N,a, \mu, p)$ is strictly decreasing as a function of $\mu$ when $p > 2 +\frac 4N$. Thus it is possible to choose a  $\beta_1 = \beta_1(a_1,a_2) >0$  such} 
\begin{align} \label{defb0}
\begin{split}
&{l(N, a_1, \mu_1 + \beta , p_1)+ l(N, a_2, \mu_2 + \beta, p_2) - \beta c_1a_1-\beta c_2a_2} \\
&{> \max\{l(N, a_1, \mu_1, p_1), l(N, a_2, \mu_2, p_2)\}}
\end{split}
\end{align}
for any $0 < \beta < \beta_1$. Note that we can choose $\beta_1 \to \infty$ as $a_1,a_2 \to 0$. Also choosing if necessary $\beta_0 >0$ smaller in Lemma \ref{mpgeo1} we can assume that $\beta_1 = \beta_0$.

\begin{lem}\label{nonzero}
For any $0 < \beta < \beta_0$,
\begin{align*}
&\inf \{J(u_1, u_2): (u_1, u_2) \in \mathcal {P}(N, a_1, \mu_1 + \beta, p_1) \times \mathcal {P}(N, a_2, \mu_2 + \beta, p_2)\}\\
&> \max\{l(N, a_1, \mu_1, p_1), l(N, a_2, \mu_2, p_2)\}.
\end{align*}
\end{lem}
\begin{proof}
For any $(u_1, u_2) \in \mathcal {P}(N, a_1, \mu_1 + \beta, p_1) \times \mathcal {P}(N, a_2, \mu_2 + \beta, p_2)$, we have
\begin{align*}
J(u_1, u_2) & = I_{\mu_1, p_1}(u_1) + I_{\mu_2, p_2}(u_2) - \beta \int_{\R^N}|u_1|^{r_1}|u_2|^{r_2}\, dx \\
            & \geq I_{\mu_1, p_1}(u_1) + I_{\mu_2, p_2}(u_2) -
            \beta \sum_{i =1}^2 \int_{\R^N}|u_i|^{r_1 + r_2}\, dx \\
            & \geq I_{\mu_1, p_1}(u_1) + I_{\mu_2, p_2}(u_2) -
            \beta \sum_{i =1}^2 \int_{\R^N} c_i |u_i|^2 +  \frac {1}{p_i}|u_i|^{p_i}\, dx \\
            & = I_{\mu_1 + \beta, p_1}(u_1) + I_{\mu_2 + \beta, p_2}(u_2) - \beta c_1a_1-\beta c_2a_2\\
            & {\geq \inf_{u \in \mathcal{P}(a_1, \mu_1 + \beta, p_1)}I_{\mu_1 + \beta, p_1}(u)
            + \inf_{v \in \mathcal{P}(a_2, \mu_2 + \beta, p_2)}I_{\mu_2 + \beta, p_2}(v)} - \beta c_1a_1-\beta c_2a_2\\
            & = l(N, a_1, \mu_1 + \beta , p_1)+ l(N, a_2, \mu_2 + \beta, p_2) -\beta c_1a_1-\beta c_2a_,
\end{align*}
where $c_i$ for $i=1, 2$ are defined by \eqref{defc}. {Thus from \eqref{defb0}, it ends the proof.}
\end{proof}

Now for any given $ \beta \in (0,  \beta_0)$, according to Lemma \ref{nonzero}, we can fix {an $\eps >0$} such that
\begin{align} \label{defeps}
\begin{split}
&\inf \{J(u_1, u_2): (u_1, u_2) \in \mathcal {P}(N, a_1, \mu_1 + \beta, p_1) \times \mathcal {P}(N, a_2, \mu_2 + \beta, p_2)\} \\
&> \max\{l(N, a_1, \mu_1, p_1), l(N, a_2, \mu_2, p_2)\} + \eps.
\end{split}
\end{align}
We set
\begin{align} \label{w}
w_1 := w_{a_1, \mu_1 + \beta, p_1}, \,\, w_2 := w_{a_2, \mu_2 + \beta, p_2},
\end{align}

From these definitions and as in \cite[Lemma 3.3]{BJS}, one obtains the following result.
{
\begin{lem} \label{phipsi}
 For $i=1,2$, there exists $\rho_i <0$ and $R_i >0$ such that
\begin{itemize}
\item[(i)] $0 < I_{\mu_i, p_i}(\rho_i*w_i) < \eps$ and $I_{\mu_i, p_i}(R_i*w_i) \leq 0$, where $\eps>0$ is determined in \eqref{defeps};
\item[(ii)] $ \displaystyle \frac{\partial}{\partial s}I_{\mu_i +\beta, p_i}(s*w_i) >0 $ for $s = \rho_i $ and $\displaystyle \frac{\partial}{\partial s}I_{\mu_i +\beta, p_i}(s*w_i)   <0$ for $s =R_i$.
\end{itemize}
\end{lem}}
Let $M := [\rho_1, R_1] \times [\rho_2, R_2]$, and for $(t_1, t_2) \in M$,
$$
g_0(t_1, t_2): =(t_1*w_1, t_2*w_2) \in S(a_1, a_2).
$$
We now introduce the min-max class
$$
\Gamma:=\{g \in C(M, S(a_1, a_2)): g= g_0 { \, \, on \, \, {\partial M}}\}.
$$

\begin{lem}\label{abbdness}
If $g \in \Gamma$, there holds
$$
\sup_{\partial M}J(g) < \max\{{ l(N, a_1, \mu_1, p_1), l(N, a_2, \mu_2, p_2)}\} + \eps.
$$
\end{lem}
\begin{proof}\
In view of Lemma \ref{phipsi}, {for $t_1 \in [\rho_1, R_1]$,}
\begin{align} \label{est}
\begin{split}
J(t_1 * w_1, \rho_2 * w_2) & \leq I_{\mu_1, p_1}(t_1 * w_1) + I_{\mu_2, p_2}(\rho_2 * w_2)\\
                            & \leq I_{\mu_1, p_1}(t_1 * w_1) +\eps
                           \leq \sup_{s \in \R} I_{\mu_1, p_1}(s * w_1) + \eps\\
                             &= \left(\frac{\mu_1 + \beta}{\mu_1}\right)^{\frac{4}{4 - N(p_1-2)}}
                             \sup_{s \in \R} I_{\mu_1, p_1}(s * w_{a_1, \mu_1, p_1}) + \eps \\
                           & \leq l(N, a_1, \mu_1, p_1) + \eps.
\end{split}
\end{align}
{Note that in \eqref{est} the equality is obtained by using \eqref{normw}}. Consequently, for $ t_1 \in [\rho_1, R_1]$, { we deduce from \eqref{est} that}
$$
J(t_1 * w_1, \rho_2 * w_2) \leq l(N, a_1, \mu_1 , p_1) + \eps,
$$
and in a similar way, for $t_2 \in [\rho_2, R_2]$,
$$
J(\rho_1 * w_1, t_2 * w_2) \leq l(N, a_2, \mu_2, p_2) + \eps.
$$
On the other hand, one can show using Lemma \ref{phipsi} that for $t_1 \in [\rho_1, R_1]$,
\begin{align*}
J(t_1*w_1, R_2*w_2) &\leq I_{\mu_1, p_1}(t_1 * w_1) + I_{\mu_2, p_2}(R_2 * w_2) \\
&\leq \sup_{s \in \R} I_{\mu_1, p_1}(s * w_1)
\leq l(N, a_1, \mu_1, p_1).
\end{align*}
Analogously, for $t_2 \in [\rho_2, R_2]$,
$
J(R_1*w_1, t_2*w_2) \leq  {l(N, a_2, \mu_2, p_2)}
$
and the lemma follows.
\end{proof}

\begin{lem}\label{linking}
For every $g \in \Gamma,$ there exists $(t_1, t_2) \in M$ such that $g(t_1, t_2) \in \mathcal{P}(N, a_1, \mu_1+\beta, p_1) \times \mathcal{P}(N, a_2, \mu_2+\beta, p_2)$.
\end{lem}
\begin{proof}
Let $g \in \Gamma$ be arbitrary, we write $g(t_1, t_2) :=(g_1(t_1, t_2), g_2(t_1, t_2))$, and we introduce the map $F_{g}: M \rightarrow \R^2$ as,
$$
F_{g}(t_1, t_2):= \left(\frac{\partial}{\partial s}I_{\mu_1 + \beta, p_1}(s*g_1(t_1, t_2))|_{s=0},
                            \frac{\partial}{\partial s}I_{\mu_2 + \beta, p_2}(s*g_2(t_1, t_2))|_{s=0}
                        \right).
$$
Since
\begin{align*}
\frac{\partial}{\partial s}I_{\mu_i + \beta, p_i}(s*g_i(t_1, t_2))|_{s=0}
=\int_{\R^N}|\nabla g_i(t_1, t_2)|^2 \, dx - \frac{\mu_i}{p_i}\left(\frac {p_i} {2} -1\right)N \int_{\R^N}|g_i(t_1, t_2)|^{p_i} \, dx.
\end{align*}
we deduce that $F_g(t_1,t_2) = (0,0)$ if and only if $g(t_1,t_2) \in \mathcal{P}(N, a_1, \mu_1+\beta, p_1) \times \mathcal{P}(N, a_2, \mu_2+\beta, p_2)$. To show that $F_{g}(t_1, t_2)= 0$ has a solution we can exactly follow the proof given in \cite[Lemma 3.5]{BJS}.
\end{proof}

At this point, we know from Lemmas \ref{nonzero}, \ref{abbdness} and  \ref{linking}, that there exists a Palais-Smale sequence for $J$ restricted to $S(a_1,a_2)$ at the level
\begin{align}\label{gam}
c(a_1, a_2):= \inf_{g \in \Gamma} \max_{(t_1, t_2) \in M}J(g(t_1, t_2))
           > \max \{l(N, a_1, \mu_1, p_1), l(N, a_2, \mu_2, p_2)\}.
\end{align}
In addition, arguing as in the proof of Theorem \ref{thm1}(ii), we obtain the following result.

\begin{lem} \label{linkps}
For any $0 < \beta < \beta_0$, there exists a Palais-Smale sequence $\{(u_1^n, u_2^n)\} \subset S_{rad}(a_1,a_2)$ for $J$ restricted to $S_{rad}(a_1,a_2)$ at the level $c(a_1, a_2)$,
which satisfies  $(u_1^n)^- \to 0$, $(u_2^n)^- \to 0$ in $H^1(\R^N)$ and the property  $Q(u_1^n, u_2^n) \to 0$ as $n \to \infty$.
\end{lem}

\begin{proof}[Proof of Theorem \ref{thm2}$(ii)$]
Let $\{(u_1^n, u_2^n)\} \subset S_{rad}(a_1,a_2)$ be given by Lemma \ref{linkps}. Then there exists $u_1, u_2 \geq 0$ such that, up to a subsequence, $(u_1^n, u_2^n) \rightharpoonup (u_1, u_2)$ in $H^1(\R^N) \times  H^1(\R^N)$ and $(u_1^n, u_2^n) \rightarrow (u_1, u_2)$ in $L^p(\R^N) \times L^p(\R^N)$ for $2 < p < 2^*$, $N \geq 2$. It follows as before that $(u_1, u_2)$ is a weak solution to \eqref{sys1} for some $(\lambda_1, \lambda_2) \in \R^2$ and thus $Q(u_1, u_2)=0$. Since $Q(u_1^n, u_2^n) = o(1)$, we deduce that
$ \int_{\R^N} |\nabla u_1^n|^2 + |\nabla u_2^n|^2 \, dx \to \int_{\R^N} |\nabla u_1|^2 + |\nabla u_2|^2 \, dx$. There results that
$J(u_1, u_2)=c(a_1, a_2)>0 $ and in particular $(u_1, u_2) \neq (0, 0)$. It remains to prove that $(u_1, u_2)\in S(a_1, a_2)$. From Lemma \ref{la}, we may suppose $\lambda_1 <0,$ and thus
$u_1 \in S(a_1).$ If $\lambda_2 <0$ we also have that $u_2 \in S(a_2)$. If we assume $\lambda_2 \geq 0$, then
$$
- \Delta u_2 = \lambda_2 u_2 + \mu_2 u_2^{p_2-1} + \beta r_2 u_1^{r_1}u_2^{r_2 -1} \geq 0,
$$
and applying Lemma \ref{liouville}, it follows that $u_2 =0$. Therefore $Q(u_1, 0)=0$, namely $u_1 \in \mathcal{P}(N, a_1, p_1, \mu_1)$ and this implies that
$$
c(a_1, a_2)= J(u_1, 0) = \frac 12 \int_{\R^N}|\nabla u_1|^2 \,dx - \frac{\mu_1}{p_1} \int_{\R^N}|u_1|^{p_1}\,dx
           = l(N, a_1, \mu_1, p_1),
$$
in contradiction with \eqref{gam}. Knowing that $(u_1,u_2) \in S(a_1,a_2)$, we conclude as previously.
\end{proof}

\section{Appendix}\label{app}
\begin{proof}[Proof of Lemma \ref{lem124}]

To begin with, we set for $i=1 ,2$,
\begin{equation*}
a:= \int_{\R^N} |\nabla u_1|^2 + |\nabla u_2|^2 \, dx, \quad
b_i:= \frac{\mu_i}{p_i} \int_{\R^N} |u_i|^{p_i} \,dx \quad
c:=  \beta \int_{\R^N} |u_1|^{r_1}|u_2|^{r_2} \, dx,
\end{equation*}
Thus defining, for $t > 0$, $\theta(t):= J(u_1^t, u_2^t)$, we then have
\begin{equation}\label{theta}
\theta(t) :=  a\, \frac{t^2}{2}  - \sum_{i =1}^2  \, b_i t^{\tilde{p}_i} - c \,  t^r,
\end{equation}
where we have set, for $i=1,2$,
$$ \tilde{p}_i := \Big(\frac{p_i}{2}-1\Big) N, \quad \mbox{and} \quad  r:= \Big(\frac{r_1 + r_2}{2}-1\Big)N.$$
Note that, under $(H_0)$, $\tilde{p}_i \in (0,1)$ if $2 <p_i < 2 + \frac{2}{N}$,  $\tilde{p}_i \in (1,2)$ if $p_i > 2 + \frac{2}{N}$, for $i =1,2$ and $r >2$.

To prove the lemma, it suffices to show that $\theta'$ admits at most two zeros on $(0 ,\infty)$. This is clearly equivalent to show that $g(t):= \frac{\theta'(t)}{t^{\alpha}}$ for $t >0$, and for a $\alpha \in \R$ to be chosen later, has at most two zeros. Note that it is not restrictive to assume that $p_1 \leq p_2$.
We have
$$
g(t) = a t^{1- \alpha} - b_1 \,  \tilde{p}_1 \,  t^{\tilde{p}_1 - 1 - \alpha} - b_2 \,  \tilde{p}_2 \, t^{\tilde{p}_2 - 1 - \alpha} - c \, r \, t^{r-1 - \alpha}.
$$
Thus
\begin{align*}
g'(t)  &=  a\, \big(1 - \alpha \big) t^{- \alpha} - b_1 \,  \tilde{p}_1 \big(\tilde{p}_1 - 1 - \alpha \big)  t^{\tilde{p}_1 - 2 - \alpha}
            \\
             &- b_2  \, \tilde{p}_2 \big(\tilde{p}_2 - 1 - \alpha \big)   \, t^{\tilde{p}_2 - 2 - \alpha} - c r\, \big(r-1 - \alpha \big)  \,  t^{r-2 - \alpha},
\end{align*}
and
\begin{align*}
g''(t)  &=  a\, (1 - \alpha)(- \alpha) t^{- \alpha -1} - b_1 \,  \tilde{p}_1 (\tilde{p}_1 - 1 - \alpha) (\tilde{p}_1 - 2 - \alpha)  t^{\tilde{p}_1 - 3 - \alpha}
            \\
             &- b_2 \,  \tilde{p}_2 (\tilde{p}_2 - 1 - \alpha) (\tilde{p}_2 - 2 - \alpha)  t^{\tilde{p}_2 - 3 - \alpha} -c \, r (r -1 - \alpha) (r-2 - \alpha) t^{r-3 - \alpha}.
\end{align*}
For convenience, we write
\begin{equation}\label{useful}
 g''(t) = \alpha_0 \, t^{- \alpha -1} -   \alpha_1 \,  t^{\tilde{p}_1 - 3 - \alpha} -  \, \alpha_2 \,  t^{\tilde{p}_2 - 3 - \alpha} -  \alpha_3 \,  t^{r-3 - \alpha},
\end{equation}
where we have set $\alpha_0:=a\, (1 - \alpha)(- \alpha)$, $\alpha_i:=b_i \,  \tilde{p}_i (\tilde{p}_i - 1 - \alpha) (\tilde{p}_i - 2 - \alpha),$ for $i=1, 2,$ and $ \alpha_3:= c \, r (r -1 - \alpha) (r-2 - \alpha).
$
We now consider the following two cases.  \medskip \\
{\it Case 1:} $ 2< p_1 \leq p_2 \leq r_1 + r_2 - \frac 2N$. If we assume that $\tilde{p}_2 \leq 1$, namely that $p_2 \leq 2 + \frac{2}{N}$,  then setting $\alpha =0$, we get that $\alpha_0 =0$, $\alpha_1  \leq 0$, $\alpha_2 \leq  0$ and $\alpha_3 >0$. Thus $g''(t) <0$, for any $t>0$ and we deduce that $g'$ is strictly decreasing. It follows that $g$ cannot have more than two zeros. Now if we assume that $\tilde{p}_2 > 1$ we choose $\alpha = \tilde{p}_2 -1 \in (0,1)$. Then $g''(t)$ becomes
$$g''(t) = \alpha_0 \,  t^{- \tilde{p}_2} -  \alpha_1 \,  t^{\tilde{p}_1 - \tilde{p}_2 -2} - \alpha_3 \,  t^{r - \tilde{p}_2}$$
with $\alpha_0 <0 $ and $\alpha_1 >0$. Also under our assumption we have $r \geq \tilde{p}_2 +1$ and we obtain that $\alpha_3 \geq 0$. Thus $g''(t) <0$ for any $t >0$, and we conclude as in the first case. \medskip \\
{\it Case 2:} $|p_1-p_2| \leq \frac 2N$. In view of the first case  we can assume that $\tilde{p}_2 >1$. We now write (\ref{useful}) as
$$ g''(t) =  t^{- \alpha -1} \Big[\alpha_0  -   \alpha_1 \, t^{\tilde{p}_1 - 2} -  \, \alpha_2 \,  t^{\tilde{p}_2 - 2 } -  \alpha_3 \,  t^{r-2}\Big] : = t^{- \alpha -1} \xi(t).$$
Let us prove that, for a convenient choice of $\alpha \leq 0$ we can insure that $\xi(t)$ is a strictly decreasing function.
Recall that we assume that $p_1 \leq p_2$. Since $|p_1-p_2| \leq \frac 2N$, it implies that $\tilde{p}_2 \leq \tilde{p}_1 + 1$ and thus we can choose a $\alpha \leq 0$ satisfying $\tilde{p}_2-2 \leq \alpha \leq  \tilde{p}_1-1 $. With this choice $\alpha_1 \leq 0, \alpha_2 \leq 0$, and $\alpha_3 >0$ because of $r >2$. It follows that $\xi$ is strictly decreasing on $(0, \infty)$.

Now having proved that $\xi$ is strictly decreasing and since $\lim_{t \to 0^+}\xi(t) >0$ and $\lim_{t \to \infty}\xi(t)=- \infty$, there exists exactly  one $t_1 > 0$ satisfying $\xi(t_1) = 0$.
Thus $g'(t)$ is strictly increasing on $(0, t_1)$, and strictly decreasing on $[t_1, \infty)$. Also we can check that $\lim_{t \to 0^+}g'(t)  <0$ and
$\lim_{t \to \infty}g'(t) = -\infty$. At this point we can  assume without restriction that
\begin{equation}\label{worst}
\max_{t > 0} g'(t) > 0.
\end{equation}
Otherwise, since $\lim_{t \to 0^+}g(t) < 0$, then $g(t) < 0$ for $t >0$, and  $g$ has no zero on $(0 , \infty)$,

 From (\ref{worst}) and the limits of $g'(t)$ we deduce that there are exactly two values $t_2 < t_3$ such that
$
g'(t_2) = g'(t_3) = 0.
$
In addition $0 < t_2 < t_1 <t_3$. Clearly $g$ is decreasing on $ (0, t_2) \cup [t_3, \infty)$, and increasing on  $[t_2, t_3)$. Recording that $\lim_{t \to 0}g(t) = 0^-$ it implies that $g$ may have at most two zeros.
\end{proof}

{\sc Address of the authors:}\\[1em]
\begin{tabular}{ll}
Tianxiang Gou & Louis Jeanjean\\
Laboratoire de Math\'ematiques (UMR 6623) & Laboratoire de Math\'ematiques (UMR 6623)\\
Universit\'{e} Bourgogne Franche-Comt\'{e} & Universit\'{e} Bourgogne Franche-Comt\'{e}\\
16, Route de Gray & 16, Route de Gray\\
25030 Besan\c{c}on Cedex & 25030 Besan\c{c}on Cedex\\
France & France\\
School of Mathematics and Statistics & louis.jeanjean@univ-fcomte.fr\\
Lanzhou University, Lanzhou, Gansu 730000\\
People's Republic of China\\
gou.tianxiang@gmail.com
\end{tabular}

\end{document}